\documentclass[a4paper,11pt]{article}
\usepackage[numbers,sort&compress]{natbib}
\usepackage{amsmath,amsthm,amsfonts,amssymb,graphicx,mathtools,iwona,float,xcolor}
\usepackage[colorlinks=true,citecolor=black,linkcolor=black,urlcolor=blue,pagebackref]{hyperref}
\usepackage[tmargin=34mm,bmargin=34mm,lmargin=26mm,rmargin=26mm]{geometry}
\usepackage{rsfso}
\usepackage[noabbrev,capitalise]{cleveref}
\makeatletter
\def\NAT@spacechar{~}
\makeatother
\renewcommand*{\backref}[1]{}
\renewcommand*{\backrefalt}[4]{\footnotesize\hspace*{0pt}\hfill \ifcase #1 \mbox{[not cited]} \or  \mbox{[p.\,#2]}  \else \mbox{[pp.\,#2]} \fi}
\theoremstyle{plain}
\newtheorem{thm}{Theorem}
\newtheorem{lem}[thm]{Lemma}
\newtheorem{cor}[thm]{Corollary}

\newtheorem{conj}[thm]{Conjecture}
\newtheorem*{claim*}{Claim}
\crefname{lem}{Lemma}{Lemmas}
\crefname{thm}{Theorem}{Theorems}
\crefname{prop}{Proposition}{Propositions}
\crefname{conj}{Conjecture}{Conjectures}
\newcommand{\arXiv}[1]{arXiv:\,\href{http://arxiv.org/abs/#1}{#1}}
\newcommand{\msn}[1]{MR:\,\href{http://www.ams.org/mathscinet-getitem?mr=MR#1}{#1}}
\newcommand{\doi}[1]{doi:\,\href{http://dx.doi.org/#1}{#1}}

\renewcommand{\baselinestretch}{1.15}
\usepackage[hang]{footmisc}

\renewcommand{\thefootnote}{\fnsymbol{footnote}}	
\allowdisplaybreaks
\newcommand\DateFootnote{
\begingroup
\renewcommand\thefootnote{}
\footnote{\today}
\setcounter{footnote}{0}
\vspace*{-3ex}
\endgroup}
\makeatletter
\renewcommand\section{\@startsection {section}{1}{\z@}{-3ex \@plus -1ex \@minus -.2ex}{2ex \@plus.2ex}{\normalfont\large\bfseries}}
\renewcommand\subsection{\@startsection{subsection}{2}{\z@}{-2.5ex\@plus -1ex \@minus -.2ex}{1.5ex \@plus .2ex}{\normalfont\normalsize\bfseries}}
\renewcommand\subsubsection{\@startsection{subsubsection}{3}{\z@}{-2ex\@plus -1ex \@minus -.2ex}{1ex \@plus .2ex}{\normalfont\normalsize\bfseries}}
 \renewcommand\paragraph{\@startsection{paragraph}{4}{\z@}{1.5ex \@plus.5ex \@minus.2ex}{-1em}{\normalfont\normalsize\bfseries}}
\renewcommand\subparagraph{\@startsection{subparagraph}{5}{\parindent}  {1.5ex \@plus.5ex \@minus .2ex}  {-1em} {\normalfont\normalsize\bfseries}}
\makeatother
\renewcommand{\thefootnote}{\fnsymbol{footnote}}	
\DeclarePairedDelimiter\ceil\lceil\rceil

\renewcommand{\geq}{\geqslant}
\renewcommand{\leq}{\leqslant}
\DeclareMathOperator{\td}{td}
\DeclareMathOperator{\ctd}{\overline{td}}
\DeclareMathOperator{\dist}{dist}
\newcommand{\bigchi}{\raisebox{1.55pt}{\scalebox{1.25}{\ensuremath\chi}}}
\newcommand{\cchi}{\bigchi_{\star}\hspace*{-0.2ex}}
\newcommand{\dchi}{\bigchi\hspace*{-0.1ex}_{\Delta}\hspace*{-0.3ex}}
\newcommand{\GG}{\mathcal{G}}
\newcommand{\HH}{\mathcal{X}}
\newcommand{\MM}{\mathcal{M}}

\newcommand{\mc}[1]{\mathcal{#1}} 
\newcommand{\CompleteTree}[1]{\ensuremath{T\langle{#1}\rangle}}
\newcommand{\Closure}[1]{\ensuremath{C\langle{#1}\rangle}}
\newcommand{\WeakClosure}[1]{\ensuremath{W\langle{#1}\rangle}}

\newcommand{\pf}[2][Proof]{\vspace{-3ex}\begin{proof}[#1]#2\end{proof}}

\begin{document}

{\Large\bfseries\boldmath\scshape Clustered Colouring in Minor-Closed Classes}

\medskip
Sergey Norin\footnotemark[2] \quad
Alex Scott\footnotemark[3] \quad
Paul Seymour\footnotemark[4] \quad
David R. Wood\footnotemark[5] 

\DateFootnote

\footnotetext[2]{Department of Mathematics and Statistics, McGill University, Montr\'eal, Canada (\texttt{snorin@math.mcgill.ca}). \\
Supported by NSERC grant 418520.}

\footnotetext[3]{Mathematical Institute,  University of Oxford, Oxford, U.K.\ (\texttt{scott@maths.ox.ac.uk}).\\ 
Supported by a Leverhulme Research Fellowship.}

\footnotetext[4]{Department of Mathematics, Princeton University, New Jersey, U.S.A. (\texttt{pds@math.princeton.edu}). \\
Supported by ONR grant N00014-14-1-0084 and NSF grant DMS-1265563.}

\footnotetext[5]{School of Mathematical Sciences, Monash University, Melbourne, Australia (\texttt{david.wood@monash.edu}). \\
Supported by the Australian Research Council.}

\emph{Abstract.} The \emph{clustered chromatic number} of a class of graphs is the minimum integer $k$ such that for some integer $c$ every graph in the class is $k$-colourable with monochromatic components of size at most $c$. We prove that for every graph $H$, the clustered chromatic number of the class of $H$-minor-free graphs is tied to the tree-depth of $H$. In particular, if $H$ is connected with tree-depth $t$ then every $H$-minor-free graph is $(2^{t+1}-4)$-colourable with monochromatic components of size at most $c(H)$. This provides the first evidence for a conjecture of Ossona de Mendez, Oum and Wood (2016) about defective colouring of $H$-minor-free graphs. If $t=3$ then we prove that 4 colours suffice, which is best possible. We also determine those minor-closed graph classes with clustered chromatic number 2. Finally, we develop a conjecture for the clustered chromatic number of an arbitrary minor-closed class. 

\bigskip
\bigskip

\hrule

\bigskip

\renewcommand{\thefootnote}{\arabic{footnote}}

\section{Introduction}


In a vertex-coloured graph, a \emph{monochromatic component} is a connected component of the subgraph induced by all the vertices of one colour. A graph $G$ is \emph{$k$-colourable with clustering} $c$ if each vertex can be assigned one of $k$ colours such that each monochromatic component has at most $c$ vertices. We shall consider such colourings, where the first priority is to minimise the number of colours, with small clustering as a secondary goal. With this viewpoint the following definition arises. The \emph{clustered chromatic number} of a graph class $\GG$, denoted by $\cchi(\GG)$, is the minimum integer $k$ such that, for some integer $c$, every graph in $\GG$ has a $k$-colouring with clustering $c$. See \citep{WoodSurvey} for a survey on clustered graph colouring. 

This paper studies clustered colouring in minor-closed classes of graphs. A graph $H$ is a \emph{minor} of a graph $G$ if a graph isomorphic to $H$ can be obtained from some subgraph of $G$ by contracting edges. A class of graphs $\MM$ is \emph{minor-closed} if for every graph $G\in\MM$ every minor of $G$ is in $\MM$, and some graph is not in $\MM$. For a graph $H$, let $\MM_H$ be the class of $H$-minor-free graphs (that is, not containing $H$ as a minor). Note that we only consider simple finite graphs. 


As a starting point, consider Hadwiger's Conjecture, which states that every graph containing no $K_{t}$-minor is properly $(t-1)$-colourable. This conjecture is easy for $t\leq 4$, is equivalent to the 4-colour theorem for $t=5$, is true for $t=6$ \citep{RST-Comb93}, and is open for $t\geq 7$. The best known upper bound on the chromatic number is $O(t\sqrt{\log t})$, independently due to \citet{Kostochka82,Kostochka84} and \citet{Thomason84,Thomason01}. This conjecture is widely considered to be one of the most important open problems in graph theory; see \citep{SeymourHC} for a survey. 

Clustered colourings of $K_t$-minor-free graphs provide an avenue for attacking Hadwiger's Conjecture. \citet{KawaMohar-JCTB07} first proved a $O(t)$ upper bound on $\cchi(\MM_{K_t})$. In particular, they proved that every $K_t$-minor-free graph is $\ceil{\frac{31}{2}t}$-colourable with clustering  $f(t)$, for some function $f$. The number of colours in this result was  improved to $\ceil{\frac{7t-3}{2}}$ by \citet{Wood10}, to $4t-4$ by  \citet*{EKKOS15}, to $3t-3$ by \citet{LO17}, and to $2t-2$ by \citet{Norin15}. Thus $\cchi(\MM_{K_t})\leq 2t-2$. See \citep{Kawa08,KO16} for analogous results for graphs excluding odd minors. For all of these results, the function $f(t)$ is very large, often depending on constants from the Graph Minor Structure Theorem. Van den Heuvel and Wood~\cite{vdHW} proved the first such result with $f(t)$ explicit. In particular, they proved that every $K_t$-minor-free graph is $(2t-2)$-colourable with clustering  $\ceil{\frac{t-2}{2}}$. The result of \citet{EKKOS15} mentioned below implies that $\cchi(\MM_{K_t})\geq t-1$.  \citet{DN17} have announced a proof that $\cchi(\MM_{K_t})=t-1$. 

Now consider the class $\MM_H$ of $H$-minor-free graphs for an arbitrary graph $H$. The maximum chromatic number of a graph in $\MM_H$ is at most $O(|V(H)|\sqrt{\log|V(H)|})$ and is at least $|V(H)|-1$ (since $K_{|V(H)|-1}$ is $H$-minor-free), and Hadwiger's Conjecture would imply that $|V(H)|-1$ is the answer. However, for clustered colourings, fewer colours often suffice. For example, \citet{DN17}  proved that graphs embeddable on any fixed surface are 4-colourable with bounded clustering, whereas the chromatic number is $\Theta(\sqrt{g})$ for surfaces of Euler genus $g$. Van den Heuvel and Wood~\cite{vdHW} proved that $K_{2,t}$-minor-free graphs are 3-colourable with clustering $t-1$, and that  $K_{3,t}$-minor-free graphs are 6-colourable with clustering $2t$. These results show that $\cchi(\MM_{H})$ depends on the structure of $H$, unlike the usual chromatic number which only depends on $|V(H)|$. 

At the heart of this paper is the following question: what property of $H$ determines  $\cchi(\MM_H)$? The following definitions help to answer this question. Let $T$ be a rooted tree. The \emph{depth} of $T$ is the maximum number of vertices on a root--to--leaf path in $T$. The \emph{closure} of $T$ is obtained from $T$ by adding an edge between every ancestor and descendent in $T$. The \emph{connected tree-depth} of a graph $H$, denoted by $\ctd(H)$, is the minimum depth of a rooted tree $T$ such that $H$ is a subgraph of the closure of $T$. This definition is a variant of the more commonly used definition of the \emph{tree-depth} of $H$, denoted by $\td(H)$, which equals the maximum connected tree-depth of the connected components of $H$. See \citep{Sparsity} for background on tree-depth. If $H$ is connected, then $\td(H)=\ctd(H)$. In fact, $\td(H)=\ctd(H)$ unless $H$ has two connected components $H_1$ and $H_2$ with $\td(H_1)=\td(H_2)=\td(H)$, in which case $\ctd(H)=\td(H)+1$. We choose to work with connected tree-depth to avoid this distinction. 

The following result is the primary contribution of this paper; it is proved in \cref{TreeDepthBounds}.

\begin{thm}
\label{Main}
For every graph $H$,  $\cchi(\MM_H)$ is tied to the (connected) tree-depth of $H$. In particular, 
$$\ctd(H)-1 \leq \cchi(\MM_H) \leq 2^{\ctd(H)+1}-4.$$
\end{thm}

The upper bound in \cref{Main} gives evidence for, and was inspired by, a conjecture of \citet*{OOW}, which we now introduce. A graph $G$ is \emph{$k$-colourable with defect} $d$ if each vertex of $G$ can be assigned one of $k$ colours so that each vertex is adjacent to at most $d$ neighbours of the same colour; that is, each monochromatic component has maximum degree at most $d$. The \emph{defective chromatic number} of a graph class $\GG$, denoted by $\dchi(\GG)$, is the minimum integer $k$ such that, for some integer $d$, every graph in  $\GG$ is $k$-colourable with defect $d$. Every colouring of a graph with clustering $c$ has defect $c-1$. Thus the defective chromatic number of a graph class is at most its clustered chromatic number. \citet{OOW} conjectured the following behaviour for the defective chromatic number of $\MM_H$.

\begin{conj}[\citep{OOW}] 
\label{tdH}
For every graph $H$, 
$$\dchi(\MM_H)=\ctd(H)-1.$$ 
\end{conj}

\citet{OOW} proved the lower bound, $\dchi(\MM_H) \geq \ctd(H)-1$, in \cref{tdH}. This follows from the observation that the closure of the rooted complete $c$-ary tree of depth $k$ is not $(k-1)$-colourable with clustering $c$. The lower bound in \cref{Main} follows since $\dchi\leq \cchi$ for every class. The upper bound in \cref{tdH} is known to hold in some special cases. \citet{EKKOS15} proved it if $H=K_t$; that is, $\dchi(\MM_{K_t})=t-1$, which can be thought of as a defective version of Hadwiger's Conjecture. \citet{OOW} proved the upper bound in \cref{tdH} if $\ctd(H)\leq 3$ or if $H$ is a complete bipartite graph. In particular, $\dchi(\MM_{K_{s,t}})=\min\{s,t\}$. 

\cref{Main} provides some evidence for \cref{tdH} by showing that $\dchi(\MM_H)$ and $\cchi(\MM_H)$ are bounded from above by some function of $\ctd(H)$. This was previously not known to be true.

While it is conjectured that $\dchi(\MM_H)= \ctd(H)-1$, the following lower bound, proved in \cref{LowerBoundSection}, shows that $\cchi(\MM_H)$ might be larger, thus providing some distinction between defective and clustered colourings. 

\begin{thm} 
\label{NewLowerBound}
For each $k\geq 2$, there is a graph $H_k$ with $\ctd(H_k)=\td(H_k)=k$ such that 
$$\cchi(\MM_{H_k}) \geq 2k-2.$$
\end{thm}

We conjecture an analogous upper bound:

\begin{conj}
\label{tdConjecture}
For every graph $H$,  $$\cchi(\MM_H)\leq 2\ctd(H)-2.$$ 
\end{conj}

A further contribution of the paper is to precisely determine the minor-closed graph classes with clustered chromatic number 2. This result is introduced and proved in \cref{2Colouring}. \cref{FatStarSection} studies clustered colourings of graph classes excluding so-called fat stars as a minor. This leads to a proof of 
\cref{tdConjecture} in the $\ctd(H)=3$ case. We conclude in \cref{Conjecture} with a conjecture about the clustered chromatic number of an arbitrary minor-closed class that generalises \cref{tdConjecture}.

\section{Tree-depth Bounds}
\label{TreeDepthBounds}

The main goal of this section is to prove that $\cchi(\MM_H)$ is bounded from above by some function of $\ctd(H)$. We actually provide two proofs. The first proof depends on deep results from graph structure theory and gives no explicit bound on the clustering. The second proof is self-contained, but gives a worse upper bound on the number of colours. Both proofs have their own merits, so we include both. 

Let \Closure{h,k} be the closure of the rooted complete $k$-ary tree of depth $h$. (Here each non-leaf node has exactly $k$ children.)\ 

If $r$ is a vertex in a connected graph $G$ and $V_i:=\{v\in V(G):\dist_G(v,r)=i\}$ for $i\geq 0$, then $V_0,V_1,\dots$ is called the \emph{BFS layering} of $G$ starting at $r$. 

\subsection{First Proof}

The first proof depends on the following Erd\H{o}s-P\'osa Theorem by \citet{RS-V}. For a graph $H$ and integer $p\geq 1$, let $p\,H$ be the disjoint union of $p$ copies of $H$. 

\newcommand{\blah}{\protect ; see \citep[Lemma~3.10]{RT17}}

\begin{thm}[\citep{RS-V}\blah]
\label{ErdosPosa}
For every non-empty graph $H$ with $c$ connected components and for all integers $p,w\geq 1$, 
for every graph $G$ with treewidth at most $w$ and containing no $p\, H$ minor,  
there is a set $X\subseteq V(G)$ of size at most $pwc$ such that $G - X$ has no $H$ minor. 
\end{thm}


The next lemma is the heart of our proof. 

\begin{lem}
\label{Heart}
For all integers $h,k,w\geq 1$, 
every \Closure{h,k}-minor-free graph $G$ of treewidth at most $w$ is $(2^h-2)$-colourable with clustering $kw$. 
\end{lem} 

\pf{
We proceed by induction on $h\geq 1$, with $w$ and $k$ fixed. The case $h=1$ is trivial since $\Closure{1,k}$ is the 1-vertex graph, so only the empty graph has no $\Closure{1,k}$ minor, and the empty graph is 0-colourable with clustering 0. Now assume that $h\geq 2$, the claim holds for $h-1$, and $G$ is a \Closure{h,k}-minor-free graph with treewidth at most $w$. Let $V_0,V_1,\dots$ be the BFS layering of $G$ starting at some vertex $r$. 

Fix $i\geq 1$. Then $G[V_i]$ contains no $ k\,\Closure{h-1,k}$ as a minor, as otherwise contracting $V_0\cup\dots\cup V_{i-1}$ to a single vertex gives a \Closure{h,k} minor (since every vertex in $V_i$ has a neighbour in $V_{i-1}$). Since $G$ has treewidth at most $w$, so does $G[V_i]$. By \cref{ErdosPosa} with $H=\Closure{h-1,k}$ and $c=1$, there is a set $X_i\subseteq V_i$ of size at most $kw$, such that $G[V_i\setminus X_i]$ has no $\Closure{h-1,k}$ minor. By induction, $G[V_i\setminus X_i]$ is $(2^{h-1}-2)$-colourable with clustering  $kw$. 
Use one new colour for $X_i$. Thus $G[V_i]$ is $(2^{h-1}-1)$-colourable with clustering  $kw$. 

Use disjoint sets of colours for even and odd $i$, and colour $r$ by one of the colours used for even $i$. No edge joins $V_i$ with $V_j$ for $j\geq i+2$. Thus $G$ is $(2^{h}-2)$-coloured with clustering  $kw$. 
}

To drop the assumption of bounded treewidth, we use the following result of \citet*{DDOSRSV04}, the proof of which depends on the graph minor structure theorem. 

\begin{thm}[\citep{DDOSRSV04}]
\label{DeVos}
For every graph $H$ there is an integer $w$ such that for every graph $G$ containing no $H$-minor, there is a partition $V_1,V_2$ of $V(G)$ such that $G[V_i]$ has treewidth at most $w$, for $i\in\{1,2\}$. 
\end{thm}

\cref{Heart,DeVos} imply:

\begin{lem}
\label{HeartHeart}
For all integers $h,k\geq 1$, there is an integer $g(h,k)$, such that every \Closure{h,k}-minor-free graph $G$ is $(2^{h+1}-4)$-colourable with clustering at most $g(h,k)$.
\end{lem} 

Fix a graph $H$.  By definition, $H$ is a subgraph of \Closure{\ctd(H),|V(H)|}. 
Thus every $H$-minor-free graph contains no $C(\ctd(H),|V(H)|)$-minor. Hence, \cref{HeartHeart} implies
$$\cchi(\MM_H)\leq 2^{\ctd(H)+1}-4,$$
which is the upper bound in \cref{Main}. 

Note \cref{FatStar} below improves the $h=3$ case in \cref{Heart}, which leads to a small constant-factor improvement in \cref{Main} for $h\geq 3$.


\subsection{Second Proof}

We now present our second proof that $\cchi(\MM_H)$ is bounded from above by some function of $\ctd(H)$. This proof is self-contained (not using \cref{ErdosPosa,DeVos}). 

Let $T$ be a rooted tree. Recall that the \emph{closure} of $T$ is the graph $G$ with vertex set $V(T)$, where two vertices are adjacent in $G$ if one is an ancestor of the other in $T$. The \emph{weak closure} of $T$ is the graph $G$ with vertex set $V(T)$, where two vertices are adjacent in $G$ if one is a leaf and the other is one of its ancestors. For $h,k\geq 1$, let $\CompleteTree{h,k}$ be the rooted complete $k$-ary tree of depth $h$. Let 
$\WeakClosure{h,k}$ be the weak closure of $\CompleteTree{h,k}$.

\begin{lem}
\label{WeakStrong}
For $h,k\geq 2$, the graph $\WeakClosure{h,k}$ contains \Closure{h,k-1} as a minor. 
\end{lem}

\pf{
Let $r$ be the root vertex. Colour $r$ blue. 
For each non-leaf vertex $v$, colour $k-1$ children of $v$ \emph{blue} and colour the other child of $v$ \emph{red}. 
Let $X$ be the set of blue vertices $v$ in $\CompleteTree{h,k}$, such that every ancestor of $v$ is blue. 
Note that $X$ induces a copy of $\CompleteTree{h,k-1}$ in $\CompleteTree{h,k}$. 
Let $v$ be a  non-leaf vertex in $X$. 
Let $w$ be the red child of $v$, and let $T_v$ be the subtree of $\CompleteTree{h,k}$ rooted at $w$. 
Then every leaf of $T_v$ is adjacent in $\WeakClosure{h,k}$ to $v$ and to every ancestor of $v$. 
Contract $T_v$ and the edge $vw$ into $v$. 
Now $v$ is adjacent to every ancestor of $v$ in $X$. 
Do this for each non-leaf vertex in $X$. 
Note that $T_u$ and $T_v$ are disjoint for distinct non-leaf vertices $u,v\in X$. 
Thus, we obtain \Closure{h,k-1} as a minor of $\WeakClosure{h,k}$. }

A \emph{model} of a graph $H$ in a graph $G$ is a collection $\{J_x : x \in V(H)\}$ of pairwise disjoint subtrees of $G$ such that for every $xy \in E(H)$ there is an edge of $G$ with one end in $V(J_x)$ and the other end in $V(J_y)$. Observe that a graph contains $H$ as a minor if and only if it contains a model of $H$. 

\begin{lem}
\label{Split}
For $h\geq 2$ and $k\geq 1$, if a graph $G$ contains $\WeakClosure{h,6k}$ as a minor, then $G$ contains subgraphs $G'$ and $G''$, both containing $\WeakClosure{h,k}$ as a minor, such that $|V(G')\cap V(G'')|\leq 1$. 
\end{lem}

\pf{
Let $\{J_x : x \in V(\WeakClosure{h,6k})\}$ be a model of $\WeakClosure{h,6k}$ in $G$. Let $r$ be the root vertex of $\WeakClosure{h,6k}$. We may assume that for each leaf vertex $x$ of  $\CompleteTree{h,6k}$, there is exactly one edge between $J_x$ and $J_r$. 

Let $Q$ be a tree obtained from $J_r$ by splitting vertices, where:
\begin{itemize}
\item $Q$ has maximum degree at most 3, 
\item $J_r$ is a minor of $Q$; let $\{Q_v:v \in V(J_r)\}$ be the model of $J_r$ in $Q$, 
so each edge $vw$ of $J_r$ corresponds to an edge of $Q$ between $Q_v$ and $Q_w$, 
\item there is a set $L$ of leaf vertices in $Q$, and a bijection $\phi$ from $L$ to the set of leaves of $\CompleteTree{h,6k}$, 
such that for each leaf $x$ of  $\CompleteTree{h,6k}$, if the edge between $J_x$ and $J_r$ in $G$ is incident with vertex $v$ in $J_r$, then $\phi^{-1}(x)$ is a vertex $z$ in $L\cap Q_v$, in which case we say $x$ and $z$ are \emph{associated}.
\end{itemize}

Let $L'\subseteq L$. Apply the following `propagation' process in $\CompleteTree{h,6k}$. Initially, say that the vertices in $\phi(L')$ are \emph{alive}  with respect to $L'$. For each parent vertex $y$ of leaves in $\CompleteTree{h,6k}$, if at least $2k$ of its $6k$ children are alive with respect to $L'$, then $y$ is also alive with respect to $L'$. Now propagate up $\CompleteTree{h,6k}$, so that a non-leaf vertex $y$ of $\CompleteTree{h,6k}$ is \emph{alive} if and only if at least $2k$ of its children are alive with respect to $L'$. Say $L'$ is \emph{good} if $r$ is alive with respect to $L'$.

For an edge $vw$ of $Q$ let $L_{vw}$ be the set of vertices in $L$  in the subtree of $Q-vw$ containing $v$, and let $L_{wv}$ be the set of vertices in $L$ in the subtree of $Q-vw$ containing $w$. Since $L$ is the disjoint union of $L_{vw}$ and $L_{wv}$, every leaf vertex of $\CompleteTree{h,6k}$ is in exactly one of $\phi(L_{vw})$ or $\phi(L_{wv})$. By induction, every vertex in  $\CompleteTree{h,6k}$ is alive with respect to $L_{vw}$ or $L_{wv}$ (possibly both). In particular, $L_{vw}$ or $L_{wv}$ is good (possibly both). 

Suppose that both $L_{vw}$ and $L_{wv}$ are good. Then at least $2k$ children of $r$ are alive with respect to $L_{vw}$, and at least $2k$ children of $r$ are alive with respect to $L_{wv}$. Thus there are disjoint sets $A$ and $B$, each consisting of $k$ children of $r$, where every vertex in $A$ is alive with respect to $L_{vw}$, and every vertex in $B$ is alive with respect to $L_{wv}$. We now define a set of vertices, said to be \emph{chosen} by $v$, all of which are alive with respect to $L_{vw}$. First, each vertex in $A$ is \emph{chosen} by $v$. Then for each non-leaf vertex $z$ chosen by $v$, choose $k$ children of $z$ that are also alive with respect to $L_{vw}$, and say they are \emph{chosen} by $v$. Continue this process down to the leaves of  $\CompleteTree{h,6k}$. We now define the graph $G'$, which is initially empty. For each vertex $z$ chosen by $v$, add the subgraph $J_z$ to $G'$. Furthermore, for each leaf vertex $z$ of $\CompleteTree{h,6k}$ chosen by $v$ and for each ancestor $y$ of $z$ chosen by $v$, add the edge in $G$ between $J_z$ and $J_y$ to $G'$. Define $G''$ analogously with respect to $B$ and $L_{wv}$. At this point, $G'$ and $G''$ are disjoint. 

The edge $vw$ in $Q$ either corresponds to an edge or a vertex of $J_r$. 
First suppose that $vw$ corresponds to an edge $ab$ of $J_r$, where $v$ is in $Q_a$ and $w$ is in $Q_b$. 
Let $J_r^1$ be the subtree of $J_r-ab$ containing $a$. 
Add $J_r^1$ to $G'$, plus the edge in $G$ between $J_r^1$ and $J_z$ for each leaf $z$ of  $\CompleteTree{h,6k}$ chosen by $v$. 
Similarly, let $J_r^2$ be the subtree of $J_r-ab$ containing $b$, and 
add $J_r^2$ to $G''$, plus the edge in $G$ between $J_r^2$ and $J_z$ for each leaf $z$ of  $\CompleteTree{h,6k}$ chosen by $w$. 
Observe that $G'$ and $G''$ are disjoint, and they both contain $\WeakClosure{h,k}$ as a minor, as desired. 

Now consider the case in which $vw$ corresponds to a vertex $z$ in $J_r$; that is, $v$ and $w$ are both in $Q_z$. 
Let $J_r^1$ be the subtree of $J_r$ corresponding to the subtree of $Q-vw$ containing $v$ (which includes $z$).  
Add $J_r^1$ to $G'$, plus the edge in $G$ between $J_r^1$ and $J_z$ for each leaf $z$ of  $\CompleteTree{h,6k}$ chosen by $v$. 
Similarly, let $J_r^2$ be the subtree of $J_r$ corresponding to the subtree of $Q-vw$ containing $w$ (which includes $z$).  
Add  $J_r^2$ to $G''$, plus the edge in $G$ between $J_r^2$ and $J_z$ for each leaf $z$ of  $\CompleteTree{h,6k}$ chosen by $w$. 
Observe that both $G'$ and $G''$ contain $\WeakClosure{h,k}$ as a minor, and $V(G_1)\cap V(G_2)=\{z\}$, as desired. 

We may therefore assume that for each edge $vw$ of $Q$, exactly one of $L_{vw}$ and $L_{wv}$ is good. Orient $vw$ towards $v$ if $L_{vw}$ is good, and towards $w$ if $L_{wv}$ is good. Since at most one leaf of $\CompleteTree{h,6k}$ is associated with each leaf of $Q$, each edge incident with a leaf of $Q$ is oriented away from the leaf. Since $Q$ is a tree, $Q$ contains a sink vertex $v$, which is therefore not a leaf. Let $w_1$, $w_2$ and possibly $w_3$ be the neighbours of $v$ in $Q$. Let $L_{i}$ be the set of vertices in $L$ in the subtree of $Q-vw_i$  containing $w_i$. Since $vw_i$ is oriented towards $v$, with respect to $vw_i$, the set $L_{i}$ is not good. Since no leaf of $\CompleteTree{h,6k}$ is associated with $v$, the sets $\phi(L_{1})$, $\phi(L_{2})$ and $\phi(L_{3})$ partition the leaves of $\CompleteTree{h,6k}$. Since each non-leaf vertex $y$ in $\CompleteTree{h,6k}$ has $6k$ children, $y$ is alive with respect to at least one of $L_{1}$, $L_{2}$ or $L_{3}$. In particular, at least one of $L_{1}$, $L_{2}$ or $L_{3}$ is good. This is a contradiction. 
}

\begin{thm}
\label{ColourNoWeakClosure}
Let $f(h):=\frac16 ( 4^h- 4)$ for every $h\geq 1$. 
Then there is a function $g:\mathbb{N}\times\mathbb{N}\rightarrow\mathbb{N}$ such that for every $k\geq 1$, 
every graph either contains $\WeakClosure{h,k}$ as a minor or is $f(h)$-colourable with clustering $g(h,k)$. 
\end{thm}

\pf{
We proceed by induction on $h\geq 1$. 
In the base case, $h=1$, since $\WeakClosure{1,k}$ is the 1-vertex graph, the result holds with $f(1)=g(1,k)=0$. 
Now assume that $h\geq 2$ and the result holds for $h-1$ and all $k$.

Let $G$ be a graph, which we may assume is connected. Let $V_0,V_1,\dots$ be a BFS layering of $G$.

Fix $i\geq 1$. Let $s$ be the maximum integer such that $G[V_i]$ contains $s$ disjoint subgraphs $G_1,\dots,G_s$ each containing a $\WeakClosure{h-1,\max\{1,6^{k-s}\}k}$ minor. First suppose that $s\geq k$. Then $G[V_i]$ contains $k$ disjoint subgraphs each containing a $\WeakClosure{h-1,k}$  minor. Contracting $V_0\cup\dots\cup V_{i-1}$ to a single vertex gives a $\WeakClosure{h,k}$ minor (since every vertex in $V_i$ has a neighbour in $V_{i-1}$), and we are done. Now assume that $s\leq k-1$.

If $s=0$, then $G[V_i]$ contains no $\WeakClosure{h-1,6^{k-1}k}$ minor. By induction, $G[V_i]$ is $f(h-1)$-colourable with clustering $g(h-1,6^{k-1}k)$. 

Now consider the case that $s\in[1,k-1]$. Apply  \cref{Split} to $G_j$ for each $j\in[1,r]$. Thus $G_j$ contains subgraphs $G'_j$ and $G''_j$, both containing $\WeakClosure{h-1,6^{k-s-1}k}$ as a minor, such that $|V(G'_j)\cap V(G''_j)|\leq 1$. Let $X:= \bigcup_{j=1}^s ( V(G'_j) \cap V(G''_j) ) $. Thus $|X|\leq s\leq k-1$. Let $A:=G[V_i] - \bigcup_{j=1}^s V(G'_j)$ and  $B:=G[V_i] - \bigcup_{j=1}^s V(G''_j)$. By the maximality of $s$, the subgraph $A$ contains no $\WeakClosure{h-1,6^{k-s-1}k}$ minor (as otherwise $A,G'_1,\dots,G'_s$ would give $s+1$ pairwise disjoint subgraphs satisfying the requirements). By induction, $A$ is $f(h-1)$-colourable with clustering $g(h-1,6^kk)$ since $6^{k-s-1}k\leq 6^kk$. Similarly, $B$ is $f(h-1)$-colourable with clustering $g(h-1,6^kk)$. By construction, each vertex in $G[V_i]$ is in at least one of $X$, $A$ or $B$. Use one new colour for $X$, which has size at most $s\leq k-1$. 

In both cases, $G[V_i]$ is $(2f(h-1)+1)$-colourable with clustering $\max\{g(h-1,6^kk),k-1\}$. 
Use a different set of $2f(h-1)+1$ colours for even $i$ and for odd $i$, 
and colour $r$ by one of the colours used for even $i$. 
No edge joins $V_i$ with $V_j$ for $j\geq i+2$. 
Since $f(h)= 4 f(h-1) + 2$,  $G$ is $f(h)$-colourable with clustering $g(h,k):=\max\{g(h-1,6^kk),k-1\}$. 
}

Note that the clustering function $g(h,k)$ in \cref{ColourNoWeakClosure} satisfies
$$g(h,k)\leq k6^{k6^{k6^{\reflectbox{$\ddots$}^{k6^k}}}},$$
where the number of $k$s is $h$. 

\begin{thm}
\label{SecondProof}
For every graph $H$, 
$$\cchi(\MM_H)\leq \tfrac16 ( 4^{\ctd(H)}-4).$$
\end{thm}

\pf{
Let $G$ be a graph not containing $H$ as a minor. 
By definition, $H$ is a subgraph of $\Closure{\ctd(H),|V(H)|}$. 
Thus $G$ does not contain $\Closure{\ctd(H),|V(H)|}$ as a minor. 
By \cref{WeakStrong}, $G$ does not contain $\WeakClosure{\ctd(H),|V(H)|+1}$ as a minor.
By \cref{ColourNoWeakClosure}, there is a constant $c=c(H)$, such that $G$ is 
$\frac16 (4^{\ctd(H)}-4)$-colourable with clustering at most $c$. 
}

\subsection{Lower Bound}
\label{LowerBoundSection}

We now prove \cref{NewLowerBound}, where $H_k:= \Closure{k,3}$, the closure of the complete ternary tree of depth $k$ (which has tree-depth and connected tree-depth $k$). 

\begin{lem}
\label{TernaryTree}
$\cchi(\MM_{\Closure{k,3}}) \geq 2k-2$ for $k\geq 2$.
\end{lem}

\pf{
Fix an integer $c$. We now recursively define graphs $G_k$ (depending on $c$), and show by induction on $k$ that 
$G_k$ has no $(2k-3)$-colouring with clustering $c$, and \Closure{k,3} is not a minor of $G_k$. 

For the base case $k=2$, let $G_2$ be the path on $c+1$ vertices. Then $G_2$ has no $\Closure{2,3}=K_{1,3}$ minor, and $G_2$ has no 1-colouring with clustering $c$. 

Assume $G_{k-1}$ is defined for some $k\geq 3$, that $G_{k-1}$ has no $(2k-5)$-colouring with clustering $c$, and \Closure{k-1,3} is not a minor of $G_{k-1}$. 
As illustrated in \cref{NewConstruction}, let $G_k$ be obtained from a path $(v_1,\dots,v_{c+1})$ as follows: 
for $i\in\{1,\dots,c\}$ add $2c-1$ pairwise disjoint copies of $G_{k-1}$ complete to $\{v_i,v_{i+1}\}$. 

\begin{figure}[t]
\centering

\includegraphics[width=\textwidth]{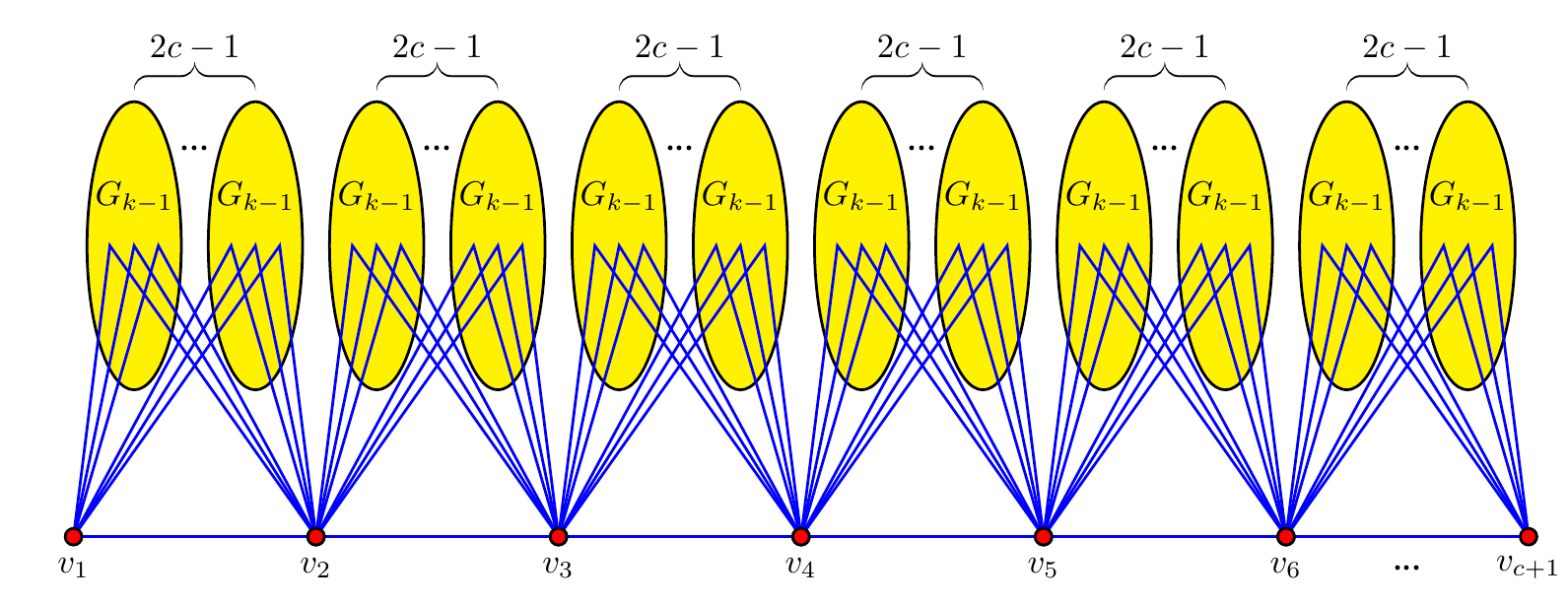}

\vspace*{-3ex}
\caption{\label{NewConstruction} Construction of $G_k$.}
\end{figure}

Suppose that $G_k$ has a $(2k-3)$-colouring with clustering $c$. 
Then $v_i$ and $v_{i+1}$ receive distinct colours for some $i\in\{1,\dots,c\}$. 
Consider the $2c-1$ copies of $G_{k-1}$ complete to $\{v_i,v_{i+1}\}$.  
At most $c-1$ such copies contain a vertex assigned the same colour as $v_i$, 
and at most $c-1$ such copies contain a vertex assigned the same colour as $v_{i+1}$. 
Thus some copy avoids both colours. 
Hence $G_{k-1}$ is $(2k-5)$-coloured with clustering $c$, which is a contradiction. 
Therefore $G_k$ has no $(2k-3)$-colouring with clustering $c$. 

It remains to show that \Closure{k,3} is not a minor of $G_k$. 
Suppose that $G_k$ contains a model $\{J_x : x \in V(\Closure{k,3})\}$ of \Closure{k,3}. 
Let $r$ be the root vertex in \Closure{k,3}. 
Choose the \Closure{k,3}-model to minimise $\sum_{x\in V(\Closure{k,3})} |V(J_x)|$. 
Since $\{v_1,\dots,v_{c+1}\}$ induces a connected dominating subgraph in $G_k$, by the minimality of the model, 
$J_r$ is a connected subgraph of $(v_1,\dots,v_{c+1})$. 
Say $J_r=(v_i,\dots,v_j)$. Note that $\Closure{k,3}-r$ consists of three pairwise disjoint copies of \Closure{k-1,3}. 
The model $X$ of one such copy avoids $v_{i-1}$ and $v_{j+1}$ (if these vertices are defined). 
Since \Closure{k-1,3} is connected,  $X$ is contained in a component of $G_k-\{v_{i-1},\dots,v_{j+1}\}$ 
and is adjacent to $(v_i,\dots,v_j)$. 
Each such component is a copy of $G_{k-1}$. 
Thus \Closure{k-1,3} is a minor of $G_{k-1}$, which is a contradiction.
Thus \Closure{k,3} is not a minor of $G_k$. 
}

\section{2-Colouring with Bounded Clustering}
\label{2Colouring}

This section considers the following question: which minor-closed graph classes have clustered chromatic number 2? To answer this question we introduce three classes of graphs that are not 2-colourable with bounded clustering, as illustrated in \cref{ThreeExamples}. 

The first example is the \emph{$n$-fan}, which is the graph obtained from the $n$-vertex path by adding one dominant vertex. If the $n$-fan is 2-colourable with clustering $c$, then the underlying path contains at most $c-1$ vertices of the same colour as the dominant vertex, implying that the other colour has at most $c$ monochromatic components each with at most $c$ vertices, and $n\leq c^2+c-1$. That is, if $n\geq c^2+c$ then the $n$-fan is not 2-colourable with clustering $c$. 

The second example is the \emph{$n$-fat star}, which is the graph obtained from the $n$-star (the star with $n$ leaves) as follows: for each edge $vw$ in the $n$-star, add $n$ degree-2 vertices adjacent to $v$ and $w$. Note that the $n$-fat star is \Closure{3,n}. Suppose that the $n$-fat star has a 2-colouring with clustering $c\leq n$. Deleting the dominant vertex in the $n$-fat star gives $n$ disjoint $n$-stars. Since $n\geq c$, in at least one of these $n$-stars, no vertex receives the same colour as the dominant vertex, implying there is a monochromatic component on $n+1\geq c+1$ vertices. Thus, for $n\geq c$ there is no 2-colouring of the  $n$-fat star with clustering $c$. 

The third example is the \emph{$n$-fat path}, which is the graph obtained from the $n$-vertex path as follows: for each edge $vw$ of the $n$-vertex path, add $n$ degree-2 vertices adjacent to $v$ and $w$. If $n\geq 2c-1$ then in every 2-colouring of the $n$-fat path with clustering $c$, adjacent vertices in the underlying path receive the same colour, implying that the underlying path is contained in a monochromatic component with more than $c$ vertices. Thus, for $n\geq 2c-1$ there is no 2-colouring of the $n$-fat path with clustering $c$. 

\begin{figure}[ht]
\centering
\includegraphics{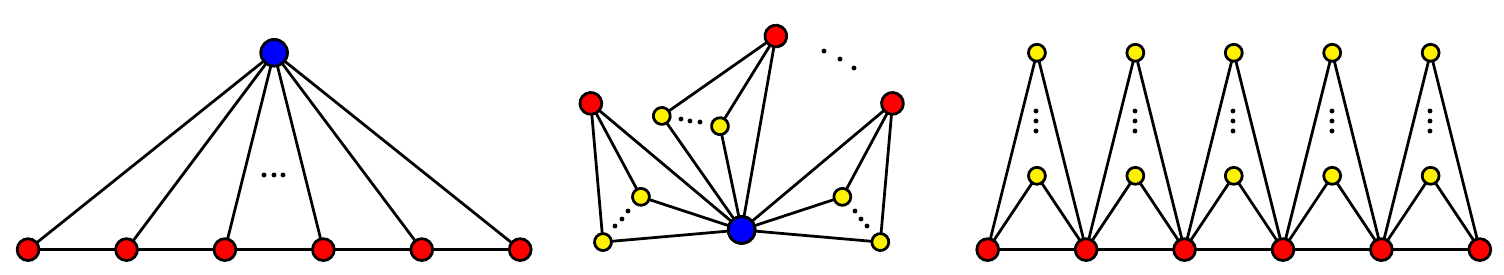}
fan \hspace*{32mm} fat star \hspace*{32mm} fat path
\caption{Graph classes that are not 2-colourable with bounded clustering.
\label{ThreeExamples}}
\end{figure}

These three examples all need three colours in a colouring with bounded clustering. The main result of this section is the following converse result. 

\begin{thm}
\label{2Colour}
Let $\GG$ be a minor-closed graph class. Then $\cchi(\GG)\leq 2$ if and only if for some integer $k\geq 2$, the $k$-fan, the $k$-fat path, and the $k$-fat star are not in $\GG$. 
\end{thm}

\cref{Quantative2Colour} below shows that every graph containing no $k$-fan minor, no $k$-fat path minor, and no $k$-fat star minor is 2-colourable with clustering $f(k)$ for some explicit function $f$. Along with the above discussion, this implies \cref{2Colour}. We assume $k\geq 2$ for the remainder of this section. 

The following definition is a key to the proof. For an $h$-vertex graph $H$ with vertex set $\{v_1,\dots,v_h\}$, a \emph{$k$-strong $H$-model} in a graph $G$ consists of $h$ pairwise disjoint connected subgraphs $X_1,\dots,X_h$ in $G$, such that for each edge $v_iv_j$ of $H$ there are at least $k$ vertices in $V(G)\setminus \bigcup_{i=1}^h V(X_i)$ adjacent to both $X_i$ and $X_j$. Note that a vertex in $V(G)\setminus \bigcup_{i=1}^h V(X_i)$ might count towards this set of $k$ vertices for distinct edges of $H$. 
This definition leads to the following sufficient condition for a graph to contain a $k$-fat star or $k$-fat path

\begin{lem} 
\label{ModelMinor}
If a graph $G$ contains a $k(k+1)$-strong $H$-model for some connected graph $H$ with $k^k$ edges, 
then $G$ contains a $k$-fat star or a $k$-fat path as a minor. 
\end{lem}

\pf{
Use the notation introduced in the definition of $k$-strong $H$-model. Since $H$ is connected with $k^k$ edges, $H$ contains a $k$-vertex path or a $k$-leaf star as a subgraph. Suppose that $(v_1,\dots,v_k)$ is a $k$-vertex path in $H$. For $i=1,2,\dots,k-1$, let $N_i$ be a set of $k+1$ vertices in $$\Big(V(G)\setminus \bigcup_{j=1}^h V(X_j)\Big)\setminus \bigcup_{j=1}^{i-1}N_j,$$ each of which is adjacent to both $X_i$ and $X_{i+1}$. Such a set exists since $X_i$ and $X_{i+1}$ have at least $k(k+1)$ common neighbours in $V(G)\setminus \bigcup_{j=1}^h V(X_j)$. For $i\in[1,k-1]$, contract one vertex of $N_i$ into $X_i$. Then contract each of $X_1,\dots,X_h$ into a single vertex. We obtain the $k$-fat path as a minor in $G$. The case of a $k$-leaf star is analogous. 
}

\begin{lem} 
\label{MakeConnected}
If a connected graph $G$ contains a $(k+2c-2)$-strong $H$-model, for some graph $H$ with $c$ connected components, then $G$ contains a $k$-strong $H'$-model for some connected graph $H'$ with $|E(H')|=|E(H)|$. 
\end{lem}

\pf{
We proceed by induction on $c\geq 1$. The case $c=1$ is vacuous. Assume $c\geq 2$, and the result holds for $c-1$. 
Let $H_1,\dots,H_c$ be the components of $H$. 
We may assume that $H$ has no isolated vertices. Say $X_1,\dots,X_h$ is a $(k+2c-2)$-strong $H$-model in $G$. 
For each edge $v_iv_j$ in $H$, let $N_{ij}$ be a set of $k+2c-2$ common neighbours of $X_i$ and $X_j$. 
For each component $H_a$ of $H$, note that $(\bigcup_{v_i\in V(H_a)} V(X_i))\cup( \bigcup_{v_iv_j\in E(H_a)} N_{ij})$ induces a connected subgraph in $G$, which we denote by $G_a$. 
Since $G$ is connected, there is a path $P$ between $G_a$ and $G_b$, for some distinct $a,b\in[1,c]$, such that no internal vertex of $P$ is in $G_1\cup \dots\cup G_c$. Note that $P$ might be a single vertex. 
For some edge $v_iv_{i'}$ in $H_a$ and some edge $v_jv_{j'}$ in $H_b$, 
without loss of generality, 
$P$ joins some vertex $x$ in $V(X_i) \cup N_{ii'}$ and some vertex $y$ in $V(X_j) \cup N_{jj'}$. 
Let $H'$ be the graph obtained from $H$ by identifying $v_i$ and $v_j$ into a new vertex $v_0$. 
Now $H'$ has $c-1$ components and  $|E(H')|=|E(H)|$. 
Define $X_0 := X_i \cup X_j \cup P$.
If $x\not\in V(X_i)$ then add the edge between $x$ and $X_i$ to $X_0$. 
Similarly, if $y\not\in V(X_j)$ then add the edge between $y$ and $X_j$ to $X_0$. 
Remove $x$ and/or $y$ from $N_{\alpha\beta}$ for each edge $v_\alpha v_\beta$ of $H'$.
Now $|N_{\alpha\beta}|\geq k + 2(c-1)-2$. 
We obtain a $(k+2(c-1)-2)$-strong $H'$-model in $G$. 
By induction, $G$ contains a $k$-strong $H''$-model for some connected graph $H''$ 
with $|E(H'')|=|E(H)|$. 
}

\begin{lem} 
\label{DisjointModelToMinor}
If a connected graph $G$ contains a $3k^k$-strong $H$-model for some graph $H$ with at least $k^k$ edges, 
then $G$ contains a $k$-fat star or a $k$-fat path as a minor. 
\end{lem}

\pf{
We may assume that $H$ has exactly $k^k$ edges and has no isolated vertices. 
Say $H$ has $c$ connected components. Then $c\leq k^k$ and $3k^k \geq k^2+k+2c-2$. 
Hence $G$ contains a $(k^2+k+2c-2)$-strong $H$-model. 
The result then follows from \cref{ModelMinor,MakeConnected}. 
}

\begin{lem}
\label{OneHigh}
Let $G$ be a connected graph such that $\deg_G(v)\geq 2\ell k$ for some non-cut-vertex $v$ and integers $k,\ell\geq 1$. Then $G$ contains a $k$-fan as a minor, or $G$ contains a connected subgraph $X$ and $v$ has $\ell$ neighbours not in $X$ and all adjacent to $X$ (thus contracting $X$ gives a $K_{2,\ell}$ minor). 
\end{lem}

\pf{
Let $r$ be a vertex of $G-v$. 
For each $w\in N_G(v)$, let $P_w$ be a $wr$-path in $G-v$. 
If $|P_w\cap N_G(v)|\geq k$ for some $w\in N_G(v)$, then $G$ contains a $k$-fan minor. 
Now assume that $|P_w\cap N_G(v)|\leq k-1$ for each $w\in N_G(v)$. 
Let $H$ be the digraph with vertex set $N_G(v)$, where 
$N^+_H(w):= V(P_w) \cap N_G(v)$ for each vertex $w$. 
Thus $H$ has maximum outdegree at most $k-1$, and the underlying undirected graph of $H$ has average degree at most $2k-2$. 
Since $|V(H)| \geq 2\ell k$, by Tur\'an's Theorem, $H$ contains a stable set $S$ of size $\ell$. 
Let $X:=\bigcup\{P_w:w\in S\}-S$, which is connected since $S$ is stable. 
Each vertex in $S$ is adjacent to $v$ and to $X$, as desired. 
}

\begin{lem}
\label{ManyHigh}
Let $G$ be a graph with distinct vertices $v_1,\dots,v_k$, such that $C:=G-\{v_1,\dots,v_k\}$ is 
connected and $\deg_C(v_i)\geq k^3$ for each $i\in[1,k]$. 
Then $G$ contains a $k$-fan or $k$-fat star as a minor.  
\end{lem}

\pf{
The idea of the proof is to attempt to build a $k$-fan model by constructing a subtree $X$ such that each $v_i$ is adjacent to a subset $S_i$ of $k$ leaves of $X$ (where the $S_i$ are disjoint).  We construct $X$ and the $S_i$ by adding, one at a time, paths to some neighbour $w$ of some $v_i$ to increase the size of $S_i$.  We always choose a neighbour at maximal distance from some root vertex, among all neighbours of all $v_i$ for which $S_i$ is not yet large enough: this ensures that later paths will not pass through the sets $S_i$ that have been previously constructed.

We now formalise this idea. 
Let $r$ be a vertex in $C$. 
Let $V_0,V_1,\dots,V_n$ be a BFS layering of $C$ starting at $r$. 
Initialise $t:=n$ and $X:=\{r\}$ and $S_i:=\emptyset$ for $i\in[1,k]$ and $S:=\emptyset$. 
The following properties trivially hold:

(0) $S=\bigcup_{i\in[1,k]} S_i$ and $S\subseteq V_t\cup V_{t+1}\cup\dots\cup V_n$.\\
(1) $X$ is a (connected) subtree of $C$ rooted at $r$ with (non-root) leaf set $S$.\\
(2) $S_i\cap S_j=\emptyset$ for distinct $i,j\in[1,k]$. \\
(3) $S_i$ is a set of at most $k+1$ neighbours of $v_i$ for $i\in[1,k]$ (and so $|S|\leq k(k+1)$). \\
(4) $|N_{C-V(X)}(v_i)| \geq k^3 -1 - (k-1) |S| > 0$ for $i\in[1,k]$. 

Now execute the following algorithm, which maintains properties (0) -- (4). Think of $V_t$ as the `current' layer. 

While $|S_i| \leq k$ for some $i\in[1,k]$ repeat the following:  
If $V_t \cap N_{C-V(X)}(v_i) =\emptyset$ for all $i\in[1,k]$ with $|S_i|\leq k$, then let $t:=t-1$. 
Properties (0) -- (4) are trivially maintained. 
Otherwise, let $w$ be a vertex in $V_t \cap N_{C-V(X)}(v_i)$ for some $i\in[1,k]$ with $|S_i|\leq k$. 
Since $V_0,V_1,\dots,V_n$ is a BFS layering of $C$ rooted at $r$ and $r$ is in $X$, 
there is a path $P$ from $w$ to $X$ consisting of at most one vertex from each of $V_0,\dots,V_t$, and with no internal vertices in $X$. 
By (0) and since $w\not\in S$,  $P$ avoids $S$. 
By (1), the endpoint of $P$ in $X$ is not a leaf of $X$. 
If $P$ contains at least $k$ vertices in $N_C(v_j)$ for some $j\in[1,k]$, then $G$ contains a $k$-fan minor and we are done. 
Now assume that $P$ contains at most $k-1$ vertices in $N_C(v_j)$ for each $j\in[1,k]$. 
Let $S_i:=S_i\cup\{w\}$ and $S:=S\cup\{w\}$ and $X:=X\cup P$. 
Now $w$ is a leaf of $X$, and property (1) is maintained. 
Properties (0), (2) and (3) are maintained by construction. 
Property (4) is maintained since $|S|$ increases by $1$ and $P$ contains at most $k-1$ vertices in $N_C(v_j)$ for each $j\in[1,k]$. 

The algorithm terminates when $|S_i|=k+1$ for each $i\in[1,k]$. Delete $C-V(X)$. Contract $X-S$ (which is connected by (1)) to a single vertex $z$. Since $S$ is the set of leaves of $X$, each vertex in $S_i$ is adjacent to both $v_i$ and $z$. Contract one edge between $v_i$ and $S_i$ for each $i\in[1,k]$. We obtain the $k$-fat star as a minor. 
}

\begin{lem}
\label{FindForest}
Let $G$ be a bipartite graph with bipartition $A,B$, such that at least $p$ vertices in 
$A$ have degree at least $k|A|$, and every vertex in $B$ has degree at least 2. 
Then $G$ contains a $k$-strong $H$-model for some graph $H$ with at least $p/2$ edges. 
\end{lem}

\pf{
Let $H$ be the graph with $V(H):=A$ where $vw\in E(H)$ whenever $|N_G(v)\cap N_G(w)| \geq k$. 
Since every vertex in $B$ has degree at least 2, 
every vertex in $A$ with degree at least $k|A|$ is incident with some edge in $H$. 
Thus $H$ has at least $p/2$ edges. 
By construction, $G$ contains a $k$-strong $H$-model.}

For the remainder of this section, let $d:= (k+2)k^k ( 18 k^{2k+1}+1 )$. 
A vertex $v$ is \emph{high-degree} if $\deg(v)\geq d$, otherwise $v$ is \emph{low-degree}. 

\begin{lem}
\label{2ConnectedHigh}
If a 2-connected graph $G$ has at least $(k+2) k^k$ high-degree vertices, 
then $G$ contains a $k$-fat path, a $k$-fat star, or a $k$-fan as a minor. 
\end{lem}

\pf{
Let $A$ be a set of exactly $(k+2) k^k$ high-degree vertices in $G$.
Let $C_1,\dots,C_p$ be the components of $G-A$. 
Say $(v,C_j)$ is a \emph{heavy pair} if $v\in A$ and $v$ has at least $6k^{k+1}$ neighbours in $C_j$. 
Since $6k^{k+1}\geq k^3$, by \cref{ManyHigh}, if some $C_j$ is in at least $k$ heavy pairs, then $G$ contains a $k$-fan or $k$-fat star as a minor, and we are done. 
Now assume that each $C_j$ is in fewer than $k$ heavy pairs. 
Let $h$ be the total number of heavy pairs. 
Then there is a set $P$ of at least $h/k$ heavy pairs containing at most one heavy pair for each component $C_j$. 
For each such heavy pair $(v,C_j)$, by \cref{OneHigh} with $\ell=3k^k$, 
$G[ V(C_j)\cup\{v\} ]$ contains a $k$-fan as a minor (and we are done) or a $K_{2,3k^k}$ minor, where $G[\{v\}]$ is the subgraph corresponding to one of the vertices in the colour class of size 2 in $K_{2,3k^k}$. 
We obtain a $3k^k$-strong $H$-model for some graph $H$, where $|E(H)|=|P|\geq h/k$. 
If $h/k\geq k^k$, then we are done by \cref{DisjointModelToMinor}. 
Now assume that $h<k^{k+1}$.
In particular, the number of vertices in $A$ that are in a heavy pair is less than $k^{k+1}$.
Let $A'$ be the set of vertices in $A$ in no heavy pair; thus $|A'| \geq 2k^k$. 
Let $H$ be the bipartite graph with bipartition $A,B$, where there is one vertex $w_j$ in $B$ for each component $C_j$, 
and $v\in A$ is adjacent to $w_j\in B$ if and only if $v$ is adjacent to some vertex in $C_j$. 
In $H$, every vertex in $A'$ has degree at least $(d-|A|)/6k^{k+1}$, which is at least $3k^k |A|$. 
(Note that $d$ is defined so that this property holds.)\ 
Since $G$ is 2-connected, each $C_j$ is adjacent to at least two vertices in $A$. 
Thus every vertex in $B$ has degree at least 2 in $H$. 
By \cref{FindForest}, $H$ contains a $3k^k$-strong model of a graph with at least $|A'|/2 \geq k^k$ edges. 
By \cref{DisjointModelToMinor} we are done. 
}

\begin{lem}
\label{MeanderingPath}
Let $V_0,V_1,\dots$ be a BFS layering in a connected graph $G$. If $G[V_i\cup V_{i+1}\cup\dots\cup V_{i+c}]$ contains a path on at least $k^{c+1}$ vertices for some $i,c\geq 0$, then $G$ contains a $k$-fan minor. 
\end{lem}

\pf{
We proceed by induction on $c$. Let $P$ be a path in $G[V_i\cup V_{i+1}\cup \dots\cup V_{i+c}]$ on $k^{c+1}$ vertices. First suppose that $P$ contains $k$ vertices $v_1,\dots,v_k$ in $V_i$ (which must happen in the base case $c=0$). Each vertex $v_i$ has a neighbour in $V_{i-1}$. Thus, contracting $G[V_0\cup\dots\cup V_{i-1}]$ into a single vertex and contracting $P$ between $v_i$ and $v_{i+1}$ to an edge (for $i\in[1,k-1]$) gives a $k$-fan minor. Now assume that $P$ contains at most $k-1 $ vertices in $V_i$ and $c\geq 1$. Thus $P-V_i$ has at least $k^{c+1}-(k-1)$ vertices and at most $k$ components. Thus some component of $P-V_i$ has at least $\ceil{(k^{c+1}-k+1)/k}=k^c$ vertices and is contained in $G[V_{i+1}\cup V_{i+2}\cup\dots\cup V_{i+c}]$. By induction, $G$ contains a $k$-fan minor. 
}


Say a vertex $v$ in a coloured graph is \emph{properly} coloured if no neighbour of $v$ gets the same colour as $v$.

\begin{lem}
\label{2Connected2Colouring}
Let $G$ be a 2-connected graph containing no $k$-fan, $k$-fat star or $k$-fat path as a minor. Let $h$ be the number of high-degree vertices in $G$. Let $r$ be a vertex in $G$. Then $G$ is 2-colourable with clustering at most $d^{k^{3(k+2)k^k}}$. Moreover, if $h=0$ then we can additionally demand that $r$ is properly coloured.
\end{lem}

\pf{
Let $V_0,V_1,\dots$ be the BFS layering of $G$ starting at $r$.

First suppose that $h=0$. Colour each vertex $v\in V_i$ by $i\bmod{2}$. 
Then $r$ is properly coloured. Every monochromatic component is contained in some $V_i$. 
Suppose that some component $X$ of $G[V_i]$ has at least $d^k$ vertices. Thus $i\geq 1$. 
Since $G$ and thus $X$ has maximum degree at most $d$, 
$X$ contains a path of $k$ vertices. Contracting $G[V_0\cup\dots\cup V_{i-1}]$ into a single vertex gives a $k$-fan minor.
This contradiction shows that the 2-colouring has clustering at most $d^k$.

Now assume that $h\geq 1$. By \cref{2ConnectedHigh}, $h \leq (k+2) k^k$. Colour all the high-degree vertices black. Let $I$ be the set of integers $i\geq 0$ such that $V_i$ contains a high-degree vertex. Colour all the low-degree vertices in $\bigcup\{V_i:i\in I\}$ white. 

Let $V_i,V_{i+1},\dots,V_{i+c}$ be a maximal sequence of layers with no high-degree vertices, where $c\geq 0$. 
Thus $V_{i-1}$ is empty or contains a high-degree vertex. Similarly,  $V_{i+c+1}$ is empty or contains a high-degree vertex. If $c$ is even, then 
colour $V_i\cup V_{i+2}\cup \dots\cup V_{i+c}$ white and 
colour  $V_{i+1}\cup V_{i+3}\cup \dots\cup V_{i+c-1}$ black. 
If $c$ is odd, then 
colour $V_i\cup V_{i+2}\cup \dots\cup V_{i+c-1}$ and $V_{i+c}$ white, and
colour $V_{i+1}\cup V_{i+3}\cup \dots\cup V_{i+c-2}$ black.
Note that if $c\geq 2$ then at least one of $V_{i+1},\dots,V_{i+c-1}$ is black. 

We now show that each black component $X$ has bounded size. 
If $X$ contains some high-degree vertex, then every vertex in $X$ is high-degree and  $|X|\leq h\leq (k+2) k^k$. 
Now assume that $X$ contains no high-degree vertices. 
Say $X$ intersects $V_j$. 
Since each black layer is preceded by and followed by a white layer,  $X$ is contained in $V_j$. 
Every vertex in $X$ has degree at most $d$ in $G$. 
Thus if $X$ has at least $d^k$ vertices, then $X$ contains a path of length $k$, 
and contracting $V_0\cup\dots\cup V_{j-1}$ to a single vertex gives a $k$-fan. 
Hence $X$ has at most $d^k$ vertices.

Finally, let $X$ be a white component. Then $X$ is contained within at most $3h \leq 3(k+2) k^k$ consecutive layers (since in the notation above, if all of $V_i,V_{i+1},\dots,V_{i+c}$ are white, then $c\leq 1$). 
Suppose that $|X| \geq d^{k^{3(k+2)k^k}}$. 
Since $X$ has maximum degree at most $d$, $X$ contains a path of length $k^{3(k+2)k^k}$. 
Thus, \cref{MeanderingPath} with $c+1= 3(k+2)k^k$ implies that  $G$ contains a $k$-fan minor. 
Hence $|X| \leq d^{k^{3(k+2)k^k}}$. 
}

We now complete the proof of \cref{2Colour}.

\begin{lem}
\label{Quantative2Colour}
Let $G$ be a graph containing no $k$-fan, no $k$-fat path, and no $k$-fat star as a minor.  
Then $G$ is 2-colourable with clustering $k d^{k^{3(k+2)k^k}}$.
\end{lem}

\pf{
We may assume that $G$ is connected. Let $r$ be a vertex of $G$. 
If $B$ is a block of $G$ containing $r$, then consider $B$ to be rooted at $r$. 
If $B$ is a block of $G$ not containing $r$, then consider $B$ to be rooted at the unique vertex in $B$ that separates $B$ from $r$. 
Say $(B,v)$ is a \emph{high-degree pair} if $B$ is a block of $G$ and $v$ has high-degree in $B$. 
Note that one vertex might be in several high-degree pairs. 

Suppose that some vertex $v$ is in at least $k$ high-degree pairs with blocks $B_1,\dots,B_k$. 
Since $d\geq 2k(k+1)$, by  \cref{OneHigh} with $\ell=k+1$, 
for $i\in[k]$, there is a connected subgraph $X_i$ in $B_i-v$ and there is a set $N_i\subseteq N_{B_i}(v)\setminus V(X_i)$ of size $k+1$, such that each vertex in $N_i$ is adjacent to $X_i$. 
For $i\in[1,k]$, contract $X_i$ into a single vertex, and contract one edge between $v$ and $N_i$.
We obtain a $k$-fat star as a minor. 
Now assume that each vertex is in fewer than $k$ high-degree pairs.

Colour each block $B$ in non-decreasing order of the distance in $G$ from $r$ to the root of $B$. 
Let $B$ be a block of $G$ rooted at $v$ (possibly equal to $r$). 
Then $v$ is already coloured in the parent block of $B$. 
Let $h_B$ be the number of high-degree pairs involving $B$. 
By \cref{2Connected2Colouring}, $B$ is 2-colourable with clustering at most $d^{k^{3(k+2)k^k}}$, 
such that if $h_B=0$ then $v$ is properly coloured. 
Permute the colours in $B$ so that the colour assigned to $v$ matches the colour assigned to $v$ by the parent block. 
Then the monochromatic component containing $v$ is contained within the parent block of $B$ along with those blocks rooted at $v$ that form a high-degree pair with $v$. As shown above, there are at most $k$ such blocks. 
Thus each monochromatic component has at most $kd^{k^{3(k+2)k^k}}$ vertices. 
}

\section{Excluding a Fat Star}
\label{FatStarSection}

This section considers colourings of graphs excluding a fat star. 
We need the following more general lemma.

\begin{lem}
\label{PlanarDefectiveClustered}
For every planar graph $H$, 
$$\cchi(\MM_H)\leq 2\,\dchi(\MM_H).$$ 
\end{lem}

\pf{
The grid minor theorem of \citet{RS-V} says that every graph in $\MM_H$ has tree-width at most some function $w(H)$. 
(\citet{CC16} recently showed that $w$ can be taken to be polynomial in $|V(H)|$.)\ 
\citet*{ADOV03} observed that every graph with tree-width $w$ and maximum degree $\Delta$ is 2-colourable with clustering $24w\Delta$. 
Let $k:=\dchi(\MM_H)$. 
That is, every $H$-minor-free graph $G$ is $k$-colourable with monochromatic components of maximum degree at most some function $d(H)$. 
Apply the above result of \citet{ADOV03} to each monochromatic component. 
Thus $G$ is $2k$-colourable with clustering $24\,w(H)\,d(H)$. 
Hence $\cchi(\MM_H) \leq 2k$. 
}

A variant of \cref{PlanarDefectiveClustered} holds for arbitrary graphs $H$ with ``2'' replaced by ``3''. 
The proof uses a result of \citet{LO17} in place of the result of \citet{ADOV03}; see \citep{EKKOS15,vdHW}. 

\begin{thm}
\label{FatStar}
For $k\geq 3$, the clustered chromatic number of the class of graphs containing no $k$-fat star minor equals 4. 
\end{thm}

\pf{
As illustrated in \cref{ThreeExamples}, the $k$-fat star is planar. \citet{OOW} proved that graphs containing no $k$-fat star minor are 2-colourable with defect $O(k^{13})$. Thus, \cref{PlanarDefectiveClustered} implies that the clustered chromatic number of the class of graphs containing no $k$-fat star is at most 4.  To obtain a bound on the clustering, note that a result of \citet{LeafSeymour15} implies that every graph containing no $k$-fat star minor has tree-width $O(k^2)$. It follows from the proof of \cref{PlanarDefectiveClustered} that every graph containing no $k$-fat star minor is 4-colourable with clustering $O(k^{15})$.  Since the $3$-fat star is \Closure{3,3}, \cref{TernaryTree} implies that for $k\geq 3$, 
the clustered chromatic number of the class of graphs containing no $k$-fat star minor is at least 4. 
}

Every graph $H$ with $\ctd(H)\leq 3$ is a subgraph of the $k$-fat star for some $k\leq |V(H)|$. 
Thus \cref{FatStar} implies  \cref{tdConjecture} in the case of connected tree-depth 3.

\begin{cor}
\label{td3}
For every graph $H$ with $\ctd(H)\leq 3$, 
$$\cchi(\MM_H)\leq  4.$$ 
\end{cor}

We can push this result further. 

\begin{thm}
\label{td3+}
For every graph $H$ with $\td(H)\leq 3$, 
$$\cchi(\MM_H)\leq 5.$$
\end{thm}

\pf{
%
Say $H$ has $p$ components. 
Each component of $H$ is a subgraph of the $k$-fat star for some $k\leq |V(H)|$. 
Let $H'$ consist of $p$ pairwise disjoint copies of the $k$-fat star. 
Let $G$ be an $H$-minor-free graph. Thus $G$ is also $H'$-minor-free. 
By the Grid Minor Theorem of \citet{RS-V} and since $H'$ is planar, 
$G$ has treewidth at most $w=w(H')$. 
By \cref{ErdosPosa}, there is a set $X$ of at most $(p-1)(w-1)$ vertices in $G$, 
such that  $G-X$ contains no $k$-fat star as a minor. 
By \cref{FatStar}, $G-X$ is 4-colourable with clustering at most some function of $H$. 
Assign vertices in $X$ a fifth colour. 
Thus $G$ is 5-colourable with clustering at most some function of $H$. 
}

\section{A Conjecture about Clustered Colouring}
\label{Conjecture}

We now formulate a conjecture about the clustered chromatic number of an arbitrary minor-closed class of graphs.  Consider the following recursively defined class of graphs. Let $\HH_{1,c}:=\{P_{c+1},K_{1,c}\}$. Here $P_{c+1}$ is the path with $c+1$ vertices, and $K_{1,c}$ is the star with $c$ leaves. As illustrated in \cref{Xkc}, for $k\geq 2$, let $\HH_{k,c}$ be the set of graphs obtained by the following three operations. For the first two operations, consider an arbitrary graph $G\in \HH_{k-1,c}$. 
\begin{itemize}
\item Let $G'$ be the graph obtained from $c$ disjoint copies of $G$ by adding one dominant vertex. Then $G'$ is in $\HH_{k,c}$. 
\item Let $G^+$ be the graph obtained from $G$ as follows: for each $k$-clique $D$ in $G$, add a stable set of $k(c-1)+1$ vertices complete to $D$. Then $G^+$ is in $\HH_{k,c}$. 
\item If $k\geq 3$ and $G\in \HH_{k-2,c}$, then let $G^{++}$ be the graph obtained from $G$ as follows: 
for each $(k-1)$-clique $D$ in $G$, add a path of $(c^2-1)(k-1)+(c+1)$ vertices complete to $D$. Then $G^{++}$ is in $\HH_{k,c}$. 
\end{itemize}

\begin{figure}[h]
\centering
\includegraphics{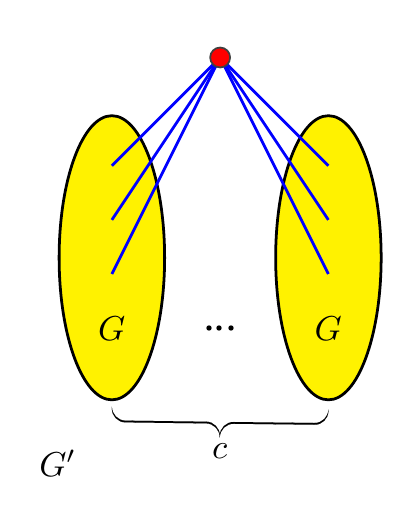}
\quad
\quad
\includegraphics{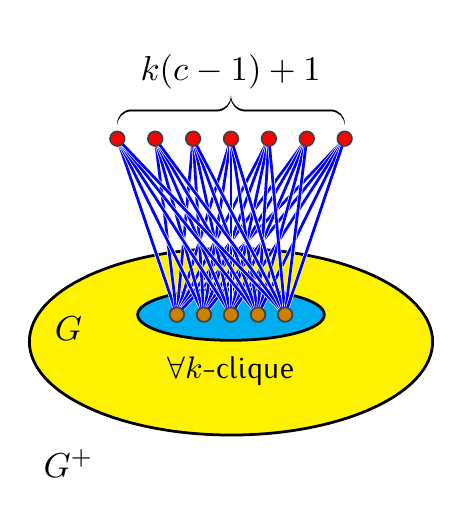}
\quad
\quad
\includegraphics{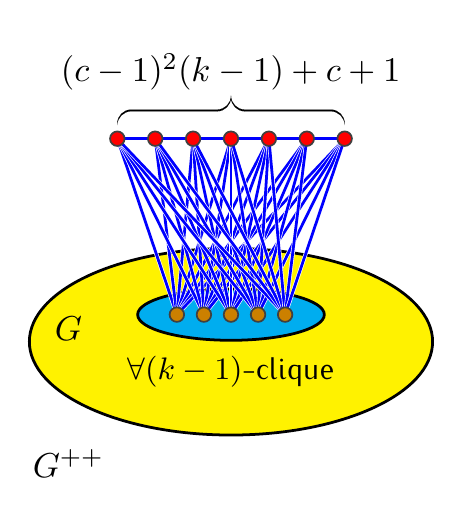}
\caption{\label{Xkc} Construction of $\HH_{k,c}$.}
\end{figure}

A vertex-coloured graph is \emph{rainbow} if every vertex receives a distinct colour. 

\begin{lem}
\label{Rainbow}
For every $c\geq 1$ and $k\geq 2$, for every graph $G\in\HH_{k,c}$, 
every colouring of $G$ with clustering $c$ contains a rainbow $K_{k+1}$.
In particular, no graph in $\HH_{k,c}$ is $k$-colourable with clustering $c$. 
\end{lem}

\pf{
We proceed by induction on $k\geq 1$. In the case $k=1$, every colouring of $P_{c+1}$ or $K_{1,c}$ with clustering $c$ contains an edge whose endpoints receive distinct colours, and we are done. Now assume the claim for $k-1$ and for $k-2$ (if $k\geq 3$). 

Let $G\in \HH_{k-1,c}$. Consider a colouring of $G'$ with clustering $c$. Say the dominant vertex $v$ is blue. At most $c-1$ copies of $G$ contain a blue vertex. Thus, some copy of $G$ has no blue vertex. By induction, this copy of $G$ contains a rainbow $K_k$. With $v$ we obtain a rainbow $K_{k+1}$. 

Now consider a colouring of $G^+$ with clustering $c$. 
By induction, the copy of $G$ in $G^+$ contains a clique $w_1,\dots,w_k$ receiving distinct colours. Let $S$ be the set of $k(c-1)+1$ vertices adjacent to $w_1,\dots,w_k$ in $G^+$. At most $c-1$ vertices in $S$ receive the same colour as $w_i$. 
Thus some vertex in $S$ receives a colour distinct from the colours assigned to $w_1,\dots,w_k$. Hence $G^+$ contains a rainbow $K_{k+1}$. 


Now suppose $k\geq 3$ and  $G\in \HH_{k-2,c}$. Consider a colouring of $G^{++}$ with clustering $c$. 
By induction, the copy of $G$ in $G^{++}$ contains a clique $w_1,\dots,w_{k-1}$ receiving distinct colours. 
Let $P$ be the path of $(c^2-1)(k-1)+(c+1)$ vertices in $G^{++}$ complete to $w_1,\dots,w_{k-1}$. 
Let $X_i$ be the set of vertices in $P$ assigned the same colour as $w_i$, and let $X:=\bigcup_iX_i$. 
Thus $|X_i|\leq c-1$ and $|X|\leq (c-1)(k-1)$. 
Hence $P-X$ has at most $(c-1)(k-1)+1$ components, and $|V(P-X)| \geq (c^2-1)(k-1)+(c+1)- (c-1)(k-1) = c\big( (c-1)(k-1)+1 \big)+1$. 
Some component of $P-X$ has at least $c+1$ vertices, and therefore contains a bichromatic edge $xy$.
Then $\{w_1,\dots,w_{k-1}\}\cup\{x,y\}$ induces a rainbow $K_{k+1}$ in $G^{++}$. 
}

We conjecture that a minor-closed class that excludes every graph in $\HH_{k,c}$ for some $c$ is $k$-colourable with bounded clustering. More precisely:

\begin{conj}
\label{MinorConjecture}
For every minor-closed class $\MM$ of graphs, 
$$\cchi(\MM)=\min\{k:\exists c\; \MM\cap \HH_{k,c}=\emptyset\}.$$
\end{conj}

Several comments about \cref{MinorConjecture} are in order (see \cref{Claims} for proofs of the following claims):

\begin{itemize}
\item To prove the lower bound in \cref{MinorConjecture}, let $k$ be the minimum integer such that $\MM\cap \HH_{k,c}=\emptyset$ for some integer $c$. Thus for every integer $c$ some graph $G\in \HH_{k-1,c}$ is in $\MM$. By \cref{Rainbow}, $G$ has no $(k-1)$-colouring with clustering $c$. Thus $\cchi(\MM)\geq k$. 

\item Note that the $k=1$ case of \cref{MinorConjecture} is trivial: a graph is 1-colourable with bounded clustering if and only if each component has bounded size, which holds if and only if every path has bounded length and every vertex has bounded degree. 

\item We note that \cref{2Colour} implies \cref{MinorConjecture} with $k=2$. 
If $G = P_{c+1}$, then $G'$ is contained in the $c(c + 1)$-fan and $G^+$ is contained in the $(2c-1)$-fat path. If $G = K_{1,c}$, then $G'$ is the $c$-fat star and $G^+$ is contained in the $(2c-1)$-fat star. It follows that if a minor-closed class $\MM$ excludes every graph in $\HH_{2,c}$ for some $c$ , then $\MM$ excludes the $c(c+1)$-fan, the $(2c-1)$-fat path, and the $(2c- 1)$-fat star. Then $\cchi(\MM)\leq 2$ by \cref{2Colour}.

\item We now relate \cref{MinorConjecture,tdConjecture}. Fix a graph $H$. \cref{MinorConjecture} says that the clustered chromatic number of $\MM_H$ equals the minimum integer $k$ such that for some integer $c$, every graph in $\HH_{k,c}$ contains $H$ as a minor.  Let  $k:= \ctd(H)\geq 2$. An easy inductive argument shows that every graph in $\HH_{2k-2,c}$ contains a \Closure{k,c} minor.  Thus, for a suitable value of $c$, every graph in $\HH_{2k-2,c}$ contains $H$ as a minor. Hence, \cref{MinorConjecture} implies \cref{tdConjecture}.

\item Consider the case of excluding the complete bipartite graph $K_{s,t}$ as a minor for $s\leq t$.  Van den Heuvel and Wood~\cite{vdHW} proved the lower bound, $\cchi(\MM_{K_{s,t}}) \geq s+1$ for $t\geq\max\{s,3\}$. Their construction is a special case of the construction above.  
We claim that \cref{MinorConjecture} asserts that $\cchi(\MM_{K_{s,t}})=s+1$ for $t\geq \max\{s,3\}$. 
To see this, first note that an easy inductive argument shows that every graph in $\HH_{s+1,t}$ contains a $K_{s,t}$ subgraph; thus $\MM_{K_{s,t}}\cap \,\HH_{s+1,t}=\emptyset$. Furthermore, another easy inductive argument shows that for all $s,c\geq1$, there is a graph in 
$\HH_{s,c}$ containing no $K_{s,\max\{s,3\}}$ minor. This implies that $\MM_{K_{s,t}} \cap \HH_{s,c} \neq \emptyset$ for all $t\geq\max\{ s,3\}$. Together these observations show that $\min\{k:\exists c\; \MM_{s,t} \cap \HH_{k,c}=\emptyset\}=s+1$ for $t\geq\max\{ s,3\}$. That is, \cref{MinorConjecture} asserts that  $\cchi(\MM_{K_{s,t}})=s+1$ for $t\geq \max\{s,3\}$. Van den Heuvel and Wood~\cite{vdHW} proved the upper bound, $\cchi(\MM_{K_{s,t}})\leq 3s$ for $t\geq s$, which was improved to $2s+2$ by \citet{DN17}.

\end{itemize}

\section{An Alternative View}

We conclude the paper with alternative versions of \cref{tdH,MinorConjecture} that shift the focus to characterising minimal minor-closed classes of given defective and clustered chromatic number.  

We start with some definitions. 
Let $H$ and $G$ be two vertex-disjoint graphs, and let $S \subseteq V(G)$. Let $G'$ be obtained from $G \cup H$ by joining every vertex of $S$ to every vertex of $H$ by an edge. Then we say that $G'$ is obtained from $G$ by  \emph{taking a join with $H$ along $S$.} Let $\mc{H}$ be a class of graphs. We say that a graph $G'$ is an \emph{$\mc{H}$-decoration} of a graph $G$, if $G'$ is obtained from $G$ by repeatedly taking joins with graphs in $\mc{H}$ along cliques of $G$. For a class of graphs $\mc{G}$, let $\mc{G} \wedge \mc{H}$ denote the class of all minors of $\mc{H}$-decorations of graphs in $\mc{G}$. One can routinely verify that the $\wedge$ operation is associative. The examples below show that it is not always commutative.

First, we introduce notation for some minor-closed classes that will serve as the basis for our constructions. Let $\mc{I}$ denote the class of graphs on at most one vertex, let $\mc{O}$ denote the class of edgeless graphs, and let $\mc{P}$ denote the class of linear forests (that is,  subgraphs of paths). Let $\mc{T}_d$ denote the class of all graphs of tree-depth at most $d$. Then $\mc{T}_1$ is a class of all edgeless graphs. It follows from the alternative definition of tree-depth given in~\cite[Section 6.1]{Sparsity}  that $\mc{T}_{d+1} = \mc{O} \wedge \mc{T}_d$. 

 The operations used in \cref{MinorConjecture} can be described as follows.
	\begin{itemize}
		\item Adding a vertex adjacent to several copies of graphs in the class $\mc{G}$ (and taking all possible minors) produces the class $\mc{I} \wedge \mc{G}$.
		\item Adding stable sets complete to cliques in graphs in  $\mc{G}$ produces the class $ \mc{G} \wedge \mc{I}$.
		\item Adding paths complete to cliques in graphs in  $\mc{G}$   produces the class $\mc{G} \wedge \mc{P}$.
\end{itemize}

Note that by definition $\mc{G} \wedge \mc{H}$ is a minor-closed class for any pair of minor-closed classes $\mc{G}$ and $\mc{H}$. 

We next present an analogue of \cref{Rainbow} using the notions introduced above. A class of graphs $\mc{G}$ is \emph{$k$-cluster rainbow} (respectively, \emph{$k$-defect rainbow}) if for every $c$ there exists $G \in \mc{G}$ such that every colouring of $G$ with clustering (respectively, defect) at most $c$ contains a rainbow clique of size $k$. For example, $\mc{I}$ is $1$-cluster rainbow and $1$-defect rainbow, $\mc{P}$ is $2$-cluster rainbow, but not $2$-defect rainbow. Note that if a class of graphs $\mc{G}$ is $k$-cluster rainbow, then clearly $\cchi(\mc{G}) \geq k$. Similarly, if  $\mc{G}$ is $k$-defect rainbow, then $\dchi(\mc{G}) \geq k$.

The proof of the following lemma parallels the proof of \cref{Rainbow}; we present it for completeness.

\begin{lem}\label{Decorate} 
Let  $\mc{G},\mc{H}$ be graph classes, such that $\mc{G}$ is $k$-cluster rainbow and $\mc{H}$ is $\ell$-cluster rainbow. Then $\mc{G} \wedge \mc{H}$ is $(k+\ell)$-cluster rainbow.
\end{lem}	

\pf{
Fix $c$, and let  $G \in \mc{G}$ and $H \in \mc{H}$ be such that every colouring of $G$   with clustering at most $c$ contains a rainbow clique of size $k$, and  every colouring of $H$   with clustering at most $c$ contains a rainbow clique of size $\ell$. Let $J$ be obtained from $G$ by taking a join of $G$ with $H$, $(c-1)k+1$ times along every clique $S$ of $G$. Then $J \in \mc{G} \wedge \mc{H}$ by definition. It remains to show that every colouring $\phi:V(J) \to C$ of $J$ for some set of colours $C$ with clustering at most $c$ contains a rainbow clique of size $k+\ell$. By the choice of $J$ there exists a clique $S$ in $G$ of size $k$, which is rainbow in $\phi$. Let $H_1,H_2,\ldots, H_r$ be copies of $H$ glued along $S$ to $G$. By the choice of $H$, for every $i$ there exists a clique $S_i$ of size $\ell$ in $H_i$ that is rainbow in $\phi$. Suppose for a contradiction that $S \cup S_i$ is not rainbow for any $i$. Then there exists $s \in S$ with a neighbour of the same colour in $S_i$  for at least $c$ choices of $i$. Thus $s$ belongs to a monochromatic  component of size at least $c+1$ in $\phi$, a contradiction.
}

Note that an analogue of \cref{Decorate} also holds for defective colourings. The proof is identical.

\medskip
Let $\mc{G}$ be a graph class obtained by taking a wedge-product of  $v$ copies of $\mc{I}$ and $p$ copies of  $\mc{P}$ in some order such that $v+2p=k+1$. Then we say that $\mc{G}$ is \emph{$k$-cluster critical}.  By \cref{Decorate} the clustered chromatic number of a $k$-cluster critical class is at least $k+1$.  (In fact, it is not difficult to see that equality holds.) For example, the class $\mc{I} \wedge \mc{P}$ of minors of fans, the class $\mc{I} \wedge \mc{I} \wedge \mc{I}$ of minors of fat stars, and the class  $\mc{P} \wedge \mc{I}$ of minors of fat paths are all possible $2$-cluster critical classes. Thus \cref{2Colour} is equivalent to the statement that $
\cchi(\mc{G}) \leq 2$ if and only if $\mc{G}$ contains no $2$-cluster critical class.

The discussion above implies that for all $k$ and $c$ every graph in $\mc{X}_{k,c}$ is a member of some $k$-cluster critical class. Conversely, for all $n,k$ there exists $c$ such that for every graph $G \in \mc{X}_{k,c}$ there exists a  $k$-cluster critical class  $\mc{G}$ such that $\mc{X}_{k,c}$ contains as minors all graphs in $\mc{G}$ on at most $n$ vertices. Thus \cref{MinorConjecture} can be reformulated as follows.
 
\begin{conj}\label{conj:cluster} Let $\mc{M}$ be a minor-closed class of graphs and $k \geq 0$ an integer. Then  $\cchi(\mc{G}) \geq k+1$ if and only if $\mc{G} \not \subseteq \mc{M}$ for some $k$-cluster critical class $\mc{G}$.
\end{conj}

Similarly, note that the $k$-term  $\wedge$-product $\wedge^k\mc{I} = \mc{I} \wedge \mc{I} \wedge \ldots \wedge \mc{I}$ is the class of minors of connected graphs of tree-depth $k$ and therefore the following conjecture is equivalent to \cref{tdH}. 

\begin{conj}\label{conj:defect} 
Let $\mc{M}$ be a minor-closed class of graphs and $k \geq 0$ an integer. Then  $\dchi(\mc{G}) \geq k+1$ if and only if $\wedge^{k+1}\mc{I} \not \subseteq \mc{M}$.
\end{conj}

\paragraph{Acknowledgement.} This research was initiated at the \emph{2017 Barbados Graph Theory Workshop} held at the Bellairs Research Institute. Thanks to the workshop participants for creating a stimulating working environment. Thanks to the referees for several instructive comments. 

  \let\oldthebibliography=\thebibliography
  \let\endoldthebibliography=\endthebibliography
  \renewenvironment{thebibliography}[1]{%
    \begin{oldthebibliography}{#1}%
      \setlength{\parskip}{0.2ex}%
      \setlength{\itemsep}{0.2ex}%
  }%
  {%
    \end{oldthebibliography}%
  }


\begin{thebibliography}{24}
\providecommand{\natexlab}[1]{#1}
\providecommand{\url}[1]{\texttt{#1}}
\providecommand{\urlprefix}{}
\expandafter\ifx\csname urlstyle\endcsname\relax
  \providecommand{\doi}[1]{doi:\discretionary{}{}{}#1}\else
  \providecommand{\doi}{doi:\discretionary{}{}{}\begingroup
  \urlstyle{rm}\Url}\fi

\bibitem[{Alon et~al.(2003)Alon, Ding, Oporowski, and Vertigan}]{ADOV03}
\textsc{Noga Alon, Guoli Ding, Bogdan Oporowski, and Dirk Vertigan}.
\newblock Partitioning into graphs with only small components.
\newblock \emph{J. Combin. Theory Ser. B}, 87(2):231--243, 2003.
\newblock \doi{10.1016/S0095-8956(02)00006-0}.
\newblock \msn{1957474}.

\bibitem[{Chekuri and Chuzhoy(2016)}]{CC16}
\textsc{Chandra Chekuri and Julia Chuzhoy}.
\newblock Polynomial bounds for the grid-minor theorem.
\newblock \emph{J. ACM}, 63(5):40, 2016.
\newblock \doi{10.1145/2820609}.
\newblock \msn{3593966}.

\bibitem[{DeVos et~al.(2004)DeVos, Ding, Oporowski, Sanders, Reed, Seymour, and
  Vertigan}]{DDOSRSV04}
\textsc{Matt DeVos, Guoli Ding, Bogdan Oporowski, Daniel~P. Sanders, Bruce
  Reed, Paul Seymour, and Dirk Vertigan}.
\newblock Excluding any graph as a minor allows a low tree-width 2-coloring.
\newblock \emph{J. Combin. Theory Ser. B}, 91(1):25--41, 2004.
\newblock \doi{10.1016/j.jctb.2003.09.001}.
\newblock \msn{2047529}.

\bibitem[{Dvo{\v{r}}{\'a}k and Norin(2017)}]{DN17}
\textsc{Zden{\v{e}}k Dvo{\v{r}}{\'a}k and Sergey Norin}.
\newblock Islands in minor-closed classes. {I}. {B}ounded treewidth and
  separators.
\newblock 2017.
\newblock \arXiv{1710.02727}.

\bibitem[{Edwards et~al.(2015)Edwards, Kang, Kim, Oum, and Seymour}]{EKKOS15}
\textsc{Katherine Edwards, Dong~Yeap Kang, Jaehoon Kim, Sang-il Oum, and Paul
  Seymour}.
\newblock A relative of {H}adwiger's conjecture.
\newblock \emph{SIAM J. Discrete Math.}, 29(4):2385--2388, 2015.
\newblock \doi{10.1137/141002177}.
\newblock \msn{3432847}.

\bibitem[{Kang and Oum(2016)}]{KO16}
\textsc{Dong~Yeap Kang and Sang-il Oum}.
\newblock Improper coloring of graphs with no odd clique minor.
\newblock 2016.
\newblock \arXiv{1612.05372}.

\bibitem[{Kawarabayashi(2008)}]{Kawa08}
\textsc{Ken{-}ichi Kawarabayashi}.
\newblock A weakening of the odd {Hadwiger's} conjecture.
\newblock \emph{Combin. Probab. Comput.}, 17(6):815--821, 2008.
\newblock \doi{10.1017/S0963548308009462}.
\newblock \msn{2463413}.

\bibitem[{Kawarabayashi and Mohar(2007)}]{KawaMohar-JCTB07}
\textsc{{Ken-ichi} Kawarabayashi and Bojan Mohar}.
\newblock A relaxed {H}adwiger's conjecture for list colorings.
\newblock \emph{J. Combin. Theory Ser. B}, 97(4):647--651, 2007.
\newblock \doi{10.1016/j.jctb.2006.11.002}.
\newblock \msn{2325803}.

\bibitem[{Kostochka(1982)}]{Kostochka82}
\textsc{Alexandr~V. Kostochka}.
\newblock The minimum {H}adwiger number for graphs with a given mean degree of
  vertices.
\newblock \emph{Metody Diskret. Analiz.}, 38:37--58, 1982.
\newblock \msn{0713722}.

\bibitem[{Kostochka(1984)}]{Kostochka84}
\textsc{Alexandr~V. Kostochka}.
\newblock Lower bound of the {H}adwiger number of graphs by their average
  degree.
\newblock \emph{Combinatorica}, 4(4):307--316, 1984.
\newblock \doi{10.1007/BF02579141}.
\newblock \msn{0779891}.

\bibitem[{Leaf and Seymour(2015)}]{LeafSeymour15}
\textsc{Alexander Leaf and Paul Seymour}.
\newblock Tree-width and planar minors.
\newblock \emph{J. Comb. Theory, Ser. {B}}, 111:38--53, 2015.
\newblock \doi{10.1016/j.jctb.2014.09.003}.
\newblock \msn{3315599}.

\bibitem[{Liu and Oum(2017)}]{LO17}
\textsc{Chun-Hung Liu and Sang-il Oum}.
\newblock Partitioning {$H$}-minor free graphs into three subgraphs with no
  large components.
\newblock \emph{J. Combin. Theory Ser. B}, 2017.
\newblock \doi{10.1016/j.jctb.2017.08.003}.

\bibitem[{Ne{\v{s}}et{\v{r}}il and Ossona~de Mendez(2012)}]{Sparsity}
\textsc{Jaroslav Ne{\v{s}}et{\v{r}}il and Patrice Ossona~de Mendez}.
\newblock \emph{Sparsity}, vol.~28 of \emph{Algorithms and Combinatorics}.
\newblock Springer, 2012.
\newblock \doi{10.1007/978-3-642-27875-4}.
\newblock \msn{2920058}.

\bibitem[{Norin(2015)}]{Norin15}
\textsc{Sergey Norin}.
\newblock Conquering graphs of bounded treewidth.
\newblock 2015.
\newblock Unpublished manuscript.

\bibitem[{Ossona~de Mendez et~al.(2017)Ossona~de Mendez, Oum, and Wood}]{OOW}
\textsc{Patrice Ossona~de Mendez, {Sang-il} Oum, and David~R. Wood}.
\newblock Defective colouring of graphs excluding a subgraph or minor.
\newblock \emph{Combinatorica}, accepted in 2017.
\newblock \arXiv{1611.09060}.

\bibitem[{Raymond and Thilikos(2017)}]{RT17}
\textsc{Jean-Florent Raymond and Dimitrios~M. Thilikos}.
\newblock Recent techniques and results on the {E}rd{\H{o}}s-{P}{\'{o}}sa
  property.
\newblock \emph{Discrete Appl. Math.}, 231:25--43, 2017.
\newblock \doi{10.1016/j.dam.2016.12.025}.
\newblock \msn{3695268}.

\bibitem[{Robertson and Seymour(1986)}]{RS-V}
\textsc{Neil Robertson and Paul Seymour}.
\newblock Graph minors. {V}. {E}xcluding a planar graph.
\newblock \emph{J. Combin. Theory Ser. B}, 41(1):92--114, 1986.
\newblock \doi{10.1016/0095-8956(86)90030-4}.
\newblock \msn{0854606}.

\bibitem[{Robertson et~al.(1993)Robertson, Seymour, and Thomas}]{RST-Comb93}
\textsc{Neil Robertson, Paul Seymour, and Robin Thomas}.
\newblock Hadwiger's conjecture for ${K}\sb 6$-free graphs.
\newblock \emph{Combinatorica}, 13(3):279--361, 1993.
\newblock \doi{10.1007/BF01202354}.
\newblock \msn{1238823}.

\bibitem[{Seymour(2015)}]{SeymourHC}
\textsc{Paul Seymour}.
\newblock Hadwiger's conjecture.
\newblock In \textsc{John Forbes~Nash Jr. and Michael~Th. Rassias}, eds.,
  \emph{Open Problems in Mathematics}, pp. 417--437. Springer, 2015.
\newblock \doi{10.1007/978-3-319-32162-2}.
\newblock \msn{MR3526944}.

\bibitem[{Thomason(1984)}]{Thomason84}
\textsc{Andrew Thomason}.
\newblock An extremal function for contractions of graphs.
\newblock \emph{Math. Proc. Cambridge Philos. Soc.}, 95(2):261--265, 1984.
\newblock \doi{10.1017/S0305004100061521}.
\newblock \msn{0735367}.

\bibitem[{Thomason(2001)}]{Thomason01}
\textsc{Andrew Thomason}.
\newblock The extremal function for complete minors.
\newblock \emph{J. Combin. Theory Ser. B}, 81(2):318--338, 2001.
\newblock \doi{10.1006/jctb.2000.2013}.
\newblock \msn{1814910}.

\bibitem[{van~den Heuvel and Wood(2018)}]{vdHW}
\textsc{Jan van~den Heuvel and David~R. Wood}.
\newblock Improper colourings inspired by {H}adwiger's conjecture.
\newblock \emph{J. London Math. Soc.}, 2018.
\newblock \doi{10.1112/jlms.12127}.
\newblock \arXiv{1704.06536}.

\bibitem[{Wood(2010)}]{Wood10}
\textsc{David~R. Wood}.
\newblock Contractibility and the {H}adwiger conjecture.
\newblock \emph{European J. Combin.}, 31(8):2102--2109, 2010.
\newblock \doi{10.1016/j.ejc.2010.05.015}.
\newblock \msn{2718284}.

\bibitem[{Wood(2018)}]{WoodSurvey}
\textsc{David~R. Wood}.
\newblock Defective and clustered graph colouring.
\newblock \emph{Electron. J. Combin.}, \#DS23, 2018.
\newblock \urlprefix\url{http://www.combinatorics.org/DS23}.
\newblock Version 1.

\end{thebibliography}

\def\soft#1{\leavevmode\setbox0=\hbox{h}\dimen7=\ht0\advance \dimen7
  by-1ex\relax\if t#1\relax\rlap{\raise.6\dimen7
  \hbox{\kern.3ex\char'47}}#1\relax\else\if T#1\relax
  \rlap{\raise.5\dimen7\hbox{\kern1.3ex\char'47}}#1\relax \else\if
  d#1\relax\rlap{\raise.5\dimen7\hbox{\kern.9ex \char'47}}#1\relax\else\if
  D#1\relax\rlap{\raise.5\dimen7 \hbox{\kern1.4ex\char'47}}#1\relax\else\if
  l#1\relax \rlap{\raise.5\dimen7\hbox{\kern.4ex\char'47}}#1\relax \else\if
  L#1\relax\rlap{\raise.5\dimen7\hbox{\kern.7ex
  \char'47}}#1\relax\else\message{accent \string\soft \space #1 not
  defined!}#1\relax\fi\fi\fi\fi\fi\fi}

\appendix
\section{Proofs of Claims in \cref{Conjecture}}
\label{Claims}

This appendix includes proofs of several claims made in the comments after \cref{MinorConjecture}.

\begin{lem}
\label{ContainedInClique}
For every graph $G$ in $\HH_{k,c}$, every clique in $G$ is contained in a $(k+1)$-clique in $G$.
\end{lem}

\begin{proof}
We proceed by induction on $k\geq 1$. The case $k=1$ holds by construction. 
Now assume that $k\geq2$ and the claim holds for smaller values of $k$. 

First, consider a clique $C$ in a graph $G'\in \HH_{k,c}$ for some $G\in \HH_{k-1,c}$. 
Let $v$ be the dominant vertex in $G'$. 
If $C$ is contained in some copy of $G$, then by induction, $C$ is contained in a $k$-clique in $G$, 
and adding $v$, we find that $C$ is contained in a $(k+1)$-clique  in $G'$. 
Otherwise $C$ includes $v$. 
By induction, $C-v$ is contained in a $k$-clique in $G$, and adding $v$, again $C$ is contained in a $(k+1)$-clique  in $G'$. 

Now consider a clique $C$ in a graph $G^+\in \HH_{k,c}$ for some $G\in \HH_{k-1,c}$. 
If $C$ is contained in $G$, then by induction, $C$ is contained in a $k$-clique in $G$. 
By construction, $C$ is contained in a $(k+1)$-clique in $G^+$. 
Otherwise $C$ includes a vertex $v$ in $G^{+}$ that is not in $G$. 
By construction, the neighbourhood of $v$ is a $k$-clique, and $C$ is contained in the neighbourhood of $v$. 
Thus $C$ is contained in a $(k+1)$-clique in $G^{++}$. 

Finally consider a clique $C$ in a graph $G^{++}\in \HH_{k,c}$ for some $G\in \HH_{k-2,c}$. 
If $C$ is contained in $G$, then by induction, $C$ is contained in a $(k-1)$-clique $D$ in $G$. 
Including two consecutive vertices in the path complete to $D$, we find that $C$ is contained in a $(k+1)$-clique in $G^{++}$. 
Otherwise $C$ includes a vertex $v$ in $G^{++}$ that is not in $G$. 
By construction, the neighbourhood of $v$ contains a $k$-clique, and $C$ is contained in the neighbourhood of $v$. 
Thus $C$ is contained in a $(k+1)$-clique in $G^{++}$. 
\end{proof}

\begin{lem} 
For $k\geq 2$, every graph in $\HH_{2k-2,c}$ contains a \Closure{k,c} subgraph.
\end{lem}

%
%
%
%
\pf{
We proceed by induction on $k\geq 2$. First consider the base case, $k=2$. 
Consider a graph $G'$, $G^+$ or $G^{++}$ in $\HH_{2,c}$.
By construction, $G'$ contains $K_{1,c}$, which is isomorphic to $\Closure{2,c}$. 
By construction, $G^+$ contains $K_{2,2c-1}$, which contains $\Closure{2,c}$. 
Note that $G^{++}$ does not apply in the $k=2$ case. 
Now assume that $k\geq 3$, and for $\ell\leq k-1$, 
every graph in $\HH_{2\ell-2,c}$ contains a \Closure{\ell,c} subgraph.
Consider a graph $G'$, $G^+$ or $G^{++}$ in $\HH_{2k-2,c}$.

First consider $G'$ for some $G\in \HH_{2k-3,c}$. 
By construction, $G$ contains some graph in $\HH_{2k-4,c}$ as a subgraph.
By induction, $G$ contains  a  \Closure{k-1,c} subgraph.
By construction, $G'$ contains  a  \Closure{k,c} subgraph.

Now consider $G^+$ for some $G\in \HH_{2k-3,c}$. 
By construction, $G$ contains some graph in $\HH_{2k-4,c}$ as a subgraph.
By induction, $G$ contains  a  \Closure{k-1,c} subgraph. 
Say $r$ is the root vertex in \Closure{k-1,c}.
For each leaf vertex $v$, the $vr$-path in the tree induces a $(k-1)$-clique in \Closure{k-1,c}, 
which is contained in a $(2k-2)$-clique $D_v$ in $G$ by \cref{ContainedInClique}. 
By construction, 
in $G^+$ there is a set $S_v$ of $k(c-1)+1\geq c$ vertices complete to $D'_v$. 
Moreover, $S_v\cap S_w=\emptyset$ for distinct leaves $v,w$. 
It follows that $G^+$ contains a $\Closure{k,c}$ subgraph.

Finally, consider $G^{++}$ for some $G\in \HH_{2k-4,c}$. 
By induction, $G$ contains  a  \Closure{k-1,c} subgraph. 
Say $r$ is the root vertex in \Closure{k-1,c}.
For each leaf vertex $v$, the $vr$-path induces a $(k-1)$-clique in \Closure{k-1,c}, 
which is contained in a $(2k-3)$-clique $D_v$ in $G$ by \cref{ContainedInClique}. 
By construction, in $G^{++}$ there is a set $S_v$ of $(c^2-1)(k-1)+ (c+1)\geq c$ vertices complete to $D_v$. 
Moreover, $S_v\cap S_w=\emptyset$ for distinct leaves $v,w$. 
It follows that $G^{++}$ contains a $\Closure{k,c}$ subgraph.
}

\begin{lem} 
For $s\geq 1$, every graph in $\HH_{s+1,t}$ contains a $K_{s,t}$ subgraph. 
\end{lem} 

\pf{
We proceed by induction on $s\geq 1$. 
Let $k=s+1$ and $c=t$. 
Consider $G'$, $G^+$ or $G^{++}$ in $\HH_{k,c}$ for some $G\in \HH_{k-1,c}$. 
Since $G\in \HH_{k-1,c}$, by \cref{Rainbow}, $G$ contains $K_{s+1}$. 
Since $k(c-1)+1 \geq t $, by construction, $G^+$ contains a $K_{s+1,t}$ subgraph. 
Since $(c^2-1)(k-1)+(c+1) \geq t $, by construction, $G^{++}$ contains a $K_{s,t}$ subgraph. 
Now consider $G'$ for some $G\in \HH_{s,t}$. 
If $s=2$ then $G'$ contains $K_{1,t}$ since $|V(G)|\geq t$. 
Now assume that $s\geq 3$. 
By induction, $G$ contains a $K_{s-1,t}$ subgraph.  
By construction, $G'$ contains a $K_{s,t}$ subgraph. 
This completes the proof. 
}

\begin{lem} 
For $s,c\geq 1$, if $t(s)= \max\{s,3\}$ then there is a $K_{s,t(s)}$-minor-free graph in $\HH_{s,c}$.
\end{lem} 

\pf{
We proceed by induction on $s\geq 1$. 
In the case $s=1$, note that $P_{c+1}$ is in $\HH_{1,c}$ and $P_{c+1}$ contains no $K_{1,3}$ minor. 
In the case $s=2$, if $G=P_{c+1}$ then $G'$ is in $\HH_{2,c}$ and $G'$ contains no $K_{2,3}$ minor (since $G'$ is outerplanar). 
In the case $s=3$, if $G''=(G')'$, then $G'$ is in $\HH_{3,c}$ and $G''$ contains no $K_{3,3}$ minor (since $G''$ is planar). 
Now assume that $s\geq 4$ and there is a $K_{s-1,t(s-1)}$-minor-free graph $G$ in $\HH_{s-1,c}$.
Consider $G'$ in $\HH_{s,c}$. 
Suppose that $G'$ contains a $K_{s,t(s)}$ minor. Then deleting the dominant vertex from $G'$, we find that $G$ contains $K_{s-1,t(s)}$ or $K_{s,t(s)-1}$ as a minor. Since $t(s)=t(s-1)+1$, in both cases, $G$ contains $K_{s-1,t(s-1)}$ as a minor. 
This contradiction shows that $G'$ contains no $K_{s,t(s)}$ minor. 
}
\end{document}